\documentclass[11pt]{article}
\usepackage{amsmath}
\usepackage{amsfonts}
\usepackage{amssymb}
\usepackage{amsthm}
\usepackage{graphicx}
\usepackage{tikz}
\usetikzlibrary{arrows,calc}
\tikzset{
>=stealth',
help lines/.style={dashed, thick},
axis/.style={<->},
important line/.style={thick},
connection/.style={thick, dotted},
}

\newtheorem{theorem}{Theorem}[section]
\newtheorem*{theorem*}{}
\newtheorem{result}[theorem]{Result}
\newtheorem{proposition}[theorem]{Proposition}
\newtheorem{lemma}[theorem]{Lemma}
\newtheorem{corollary}[theorem]{Corollary}

\theoremstyle{definition}
\newtheorem{definition}[theorem]{Definition}

\theoremstyle{remark}
\newtheorem{remark}[theorem]{Remark}

\numberwithin{equation}{section}

\newcommand{\LL}{L}
\newcommand{\LLs}{L^{\sharp}}
\newcommand{\Ss}{\mathcal{S}}
\newcommand{\D}{\mathcal{D}}
\newcommand{\E}{\mathcal{E}}

\newcommand{\tr}{\mathrm{Tr}}

\newcommand{\supp}[1]{{\text{supp}}({#1})}
\newcommand{\ints}{\mathbb{Z}}
\newcommand{\reals}{\mathbb{R}}

\newcommand{\comps}{\mathbb{C}}
\newcommand{\nats}{\mathbb{N}}

\newcommand{\dif}{\mathrm{d}}
\newcommand{\n}{\vec{n}}
\newcommand{\dist}{\mathrm{dist}}

\newcommand{\bfx}{{\bf{x}}}
\newcommand{\bfy}{{\bf{y}}}
\newcommand{\bfu}{{\bf{u}}}
\newcommand{\x}{{\bf{x}}}

\newcommand{\bfg}{\mathbf{g}}
\newcommand{\bfh}{\mathbf{h}}
\newcommand{\sfd}{\mathsf{d}}

\newcommand{\bfe}{\mathbf{e}}
\newcommand{\bfxi}{\boldsymbol{\xi}}
\newcommand{\bfeta}{\boldsymbol{\eta}}
\newcommand{\sfg}{\mathsf{G}}
\newcommand{\ran}{\mathrm{ran}}
\newcommand{\tp}{\tilde{p}}
\renewcommand{\t}{\mathrm{t}}

\newcommand{\N}{\mathcal{N}}
\newcommand{\cB}{\mathcal{B}}

\newcommand{\Op}{\mathrm{Op}}

\numberwithin{equation}{section}
\title{On a polyharmonic Dirichlet problem and boundary effects in surface spline approximation
}
\author{Thomas C. Hangelbroek\thanks{ Research supported by grant DMS-1413726 from the National
    Science Foundation.}
 }

\begin{document}
\maketitle
\begin{abstract}
For compact domains with smooth boundaries,  
we present an approximation scheme
for surface spline approximation that delivers precise 
$L_p$  approximation orders on well known smoothness spaces. 
This scheme overcomes the boundary 
effects when centers are placed with greater density near to the boundary.
It owes its success to an integral identity using a minimal number of boundary layer potentials, 
which, in turn is derived from the boundary layer potential solution to 
the Dirichlet problem for the $m$-fold Laplacian.
 Furthermore, his integral identity is shown to  be the ``native space extension'' of  the target function.
\end{abstract}
 
\section{Introduction}\label{S:intro}
In this paper we consider three seemingly unrelated problems. 
The first -- the solution of the polyharmonic Dirichlet problem
with boundary layer potentials  -- is a basic problem in potential theory and  elliptic PDE.
The second seeks a linear operator that provides smooth extensions to functions defined on bounded domains.
The third treats the complication of the boundary in surface spline approximation -- this is a fundamental problem for
kernel based approximation and is of prime importance for treating scattered data.

All three problems involve the fundamental solution 
$\phi_{m,d}:\reals^d\to \reals$ to the $m$ fold Laplacian $\Delta^m$ on $\reals^d$, for $m>d/2$. 
Before expanding on the connection between them, we 
give a more detailed explanation.

\paragraph{Problem 1: Surface Spline Approximation}
Radial basis function (RBF) approximation 
involves approximating
a target function $f$
by a 
linear combination of translates of a fixed, radially symmetric function (the radial basis function) 
$\phi:\reals^d \to \reals$ 
sampled from a finite point-set $\Xi\subset \reals^d$.
The approximant takes the form
$s_{f,\Xi}(x)  = \sum_{\xi\in\Xi} A_{\xi} \phi(x-\xi)$, where the coefficients $(A_\xi)_{\xi\in\Xi} \in \reals^{\Xi}$ are to be determined.
(For technical reasons, one often permits the addition of an auxiliary, low-degree polynomial term -- we ignore this for now, but it is 
expanded upon later on.)

A basic family of radial basis functions is the family of {\em surface splines}, 
which are (up to a constant multiple) the fundamental solutions 
$\phi_{m,d}$ of the $m$ fold Laplacian in $\reals^d$.
We consider  the approximation power of RBF approximation with surface splines over bounded regions:
when $\Omega$ is a bounded domain, $f:\Omega\to \reals$ and $\Xi\subset \Omega$.
Specifically, we wish to determine precisely the degradation
 of error estimates for surface spline approximation 
 in the presence of the boundary, and how this may be overcome.
  A detailed explanation of these ``boundary effects'' can be 
 found in Section \ref{SSapprox_background}.

\paragraph{Problem 2: Norm Minimizing Extension}
For a bounded  region $\Omega\subset\reals^d$ and $f\in W_2^{m}({\Omega})$, 
we wish to find an extension $f_e:\reals^d\to \comps$ that is best in the sense that it has
a minimal $m$th semi-norm 
$$|f_e|_{D^{-m}L_2}  := 
\left(\sum_{|\alpha| = m} \begin{pmatrix} m\\ \alpha\end{pmatrix} \int_{\reals^d}|D^{\alpha} f_e(x)|^2 \dif x\right)^{1/2}.$$
This is the $m$th Sobolev semi-norm, but in this context it is often called the Beppo-Levi semi-norm.
The Beppo-Levi space
$D^{-m}L_2 (\reals^d)= \{f\in W_{2,loc}^m(\reals^d)\mid \ |f|_{D^{-m}L_2}<\infty\} $ is a reproducing kernel semi-Hilbert space;
it and the above  extension have been studied in \cite{D2}. 
There, Duchon has shown that  
$f_e$ can be expressed as a convolution $f_e = \mu_f * \phi_{m,d}+p$, where $\mu_f$ is a distribution supported in
$\overline{\Omega}$ and $p$ is a polynomial of degree at most $m-1$. 

In \cite{MJ3}, Johnson demonstrates that the linear map $f\mapsto \mu_f$, is bounded from $W_2^m(\Omega)$  to $W_2^{-m}(\reals^d)$
and from  the Besov space $B_{2,1}^{m+1/2}(\Omega)$ to $B_{2,\infty}^{1/2-m}(\reals^d)$.
The mapping properties of this extension operator  have  been exploited in scattered data fitting problems, starting with
\cite{D2}, but continuing in \cite{MJ3,MJ4} and \cite{loehndorf}.
The more general notion of a native space extension operator for a conditionally positive definite kernel, which this is an example of,
has been introduced and studied in \cite{Schaback_NS}
 
 The goal here is to identify the distribution $\mu_f$
 explicitly in terms  of $f$ on $\overline{\Omega}$: 
 namely, in terms of values in $\Omega$ and boundary data on $\partial \Omega$. To date the only 
 case where this is known is when $m=2$ on the  disk $\Omega = B(0,1)$ in $\reals^2$ \cite{MJ1}.
 
\paragraph{Problem 3: Layer Potential Solution of a Dirichlet Problem}
For a compact region $\Omega\subset\reals^d$,  we consider the homogeneous $m$-fold Laplacian with non homogeneous boundary conditions:
\begin{equation}\label{Dirichlet}
\begin{cases} \Delta^m u(x) = 0,\quad \text{for $x\in \Omega$};\\
\lambda_k u = h_k,\quad \text{for }k=0\text{ to }m-1.\\
\end{cases}
\end{equation}
We employ the boundary differential operators
\begin{equation}
\label{bdry_operator_def}
\lambda_k :=
\begin{cases} 
\tr \Delta^{\frac{k}{2}},&\ \text{for even $j$;}\\
D_{\vec{n}} \Delta^{\frac{k-1}{2}},&\ \text{for odd $j$.}
\end{cases}
\end{equation}
(here $\tr:C(\overline{\Omega})\to C(\partial \Omega)$ is the restriction to the boundary and $\n$ is the outer unit normal to the boundary). 
Roughly, our goal is to provide a solution using $m$ boundary layer potentials
$u(x) = \sum_{j=0}^{m-1} \int_{\partial \Omega} g_j(\alpha) \lambda_{j,\alpha} \phi_{m,d}(x-\alpha)\, \dif\sigma(\alpha)+p(x)$
with an extra polynomial term $p$; 
in short, we wish to find 
 auxiliary functions $g_0\dots g_{m-1}$ given boundary data
 $h_0\dots h_{m-1}$.

\paragraph{The connection between the problems}
The solution of each of these three problems hinges on the ability to represent a function 
$f:\overline{\Omega} \to \comps$ 
with a combination of integrals of the form
\begin{equation}\label{basicrep}
f(x) = \int_{\Omega} \Delta^m f(\alpha) \phi_{m,d}(x-\alpha) \dif \alpha + \sum_{j=0}^{m-1}\int_{\partial \Omega} N_jf(\alpha) \lambda_{j,\alpha} \phi(x - \alpha)\dif \sigma(\alpha) + p(x).
\end{equation}

This representation indicates precisely the distribution $\mu_f$ used in the norm minimizing Sobolev extension (Problem 2). 
A special example of its use is to provide the solution to the Dirichlet problem (Problem 3); 
in turn, establishing the boundary layer solution of (\ref{Dirichlet}) yields almost directly the formula (\ref{basicrep}).

Finally, a certain discretization of the  representation yields an approximation scheme which conveniently addresses the boundary effects. 
This scheme replaces the kernels  appearing in (\ref{basicrep}), namely $\phi_{m,d}(x,\alpha)$ and $\lambda_{j,\alpha} \phi_{m,d}(x -\alpha)$, 
 by new kernels:
$k(x,\alpha)$ and $k_j(x,\alpha)$,
where  $k(x,\alpha) = \sum_{\xi\in\Xi} a(\alpha,\xi) \phi_{m,d}(x - \xi) $ and
$k_j(x,\alpha) = \sum_{\xi\in\Xi} a_j(\alpha,\xi) \phi_{m,d}(x - \xi) $. 
The approximant
\begin{equation}\label{scheme}
T_{\Xi}f(x) = \int_{\Omega} \Delta^m f(\alpha) k(x,\alpha) \dif \alpha + 
\sum_{j=0}^{m-1}\int_{\partial \Omega} N_j f(\alpha) 
k_j(x, \alpha)\dif \sigma(\alpha)
+p(x)
\end{equation}
generates an RBF approximant
and provides
precise approximation orders for surface spline approximation (Problem 1).
Moreover,  on certain point sets $\Xi$ it successfully treats the boundary effects by permitting rates of convergence 
matching those of  the 
boundary-free setting.
Such a scheme has been introduced in \cite{Hdisk} to treat the problem when $\Omega$ is the disk in $\reals^2$, but 
earlier
schemes of this sort have been used in \cite{DeRo,DyRo}. 
Discretizations of the form $k_j$ and $k$ were initially introduced in
\cite{DynLevinRippa}.

%
%
%
%

\subsection{Background on boundary effects for surface spline approximation}\label{SSapprox_background}
Boundary effects for surface spline approximation (as well as other RBF methods) 
can often be observed numerically \cite{FDWC,loehndorf,ZhJi}.
They can also be demonstrated analytically, by showing
that  the approximation order from finite dimensional spaces
generate by $\phi_{m,d}$
is prematurely saturated.
The meaning of this statement is explained below.

For $J\in \nats$, let 
$$S_J(\Xi,\phi):= 
\left\{  
\sum_{\xi\in\Xi} A_{\xi}  \phi(\cdot - \xi) +p
 \left| \,p\in \Pi, 
 \quad \forall q\in \Pi, 
  \sum_{\xi \in \Xi} A_{\xi} q(\xi) = 0
\right.
\right\} ;
$$
this is the space generate by $\phi$ and $\Xi$, augmented by $\Pi_J$ (polynomials of degree at most $J$)
with corresponding moment conditions on the coefficients.
The $L_p$ approximation order is defined as $\gamma>0$ so that
$$ 
\dist(f,S_J(\Xi,\phi))_p:=\min_{s\in S_J(\Xi,\phi)} \| f - s\|_{L_p} =O(h^\gamma)$$
where $h$, the {\em fill distance}
\begin{equation}\label{fill_dist}
h:=h(\Xi,\Omega):=\sup_{x\in \Omega} \dist( x,\Xi),
\end{equation}
measures the density of $\Xi$ in $\Omega$. 

The first positive results concerning approximation orders in this setting were
obtained by Duchon. In \cite{D1} and \cite{D2}  it was shown that, 
on domains satisfying an interior cone condition, interpolation of 
a function in $D^{-m}L_2(\reals^d)$
delivers $L_p$ approximation order 
$\gamma_p:= \min(m,m+d/p-d/2)$.
More precisely, 
for the (unique) function 
$I_{\Xi} f\in  S_{m-1}(\Xi,\phi_{m,d})$
which satisfies $I_{\Xi} f\left|_{\Xi}\right. = f\left|_{\Xi}\right.$,
the estimate
$\|f- I_{\Xi}f\|_p \le C h^{\gamma_p} \|f\|_{W_2^m(\Omega)}$
holds.
 This approximation order is illustrated in Figure \ref{Fi:orders}
as a dotted line.

 In \cite{MN}, Madych and Nelson introduced 
interpolation by  surface splines on multi-integer grids,
i.e., where  centers are assumed to be $h\ints^d$ and the domain of $f$ is all of 
$\reals^d$
(in this case $\Xi = h\ints^d$ is not finite and the space  $S(h\ints^d, \phi_{m,d})$ 
consists of convergent infinite 
linear combinations\footnote{Because $\phi_{k,d}$ has global support,
one considers linear combinations generated by a bounded, rapidly decaying
``localization'' 
$\psi = \sum_{j\in \ints^d} a_j \phi_{m,d}(\cdot - j)$  of shifts of $\phi_{m,d}$.}).
 Buhmann demonstrated that 
 interpolation in this setting enjoys substantially larger approximation orders than observed in the
 work of Duchon. 
In  \cite{Buh} it is
shown that interpolation by functions in $S(h\ints^d, \phi_{m,d})$ 
of shifts of $\phi_{m,d}$
 delivers approximation order $2m$ for sufficiently smooth functions.
Other ``free space'' results for surface spline approximation 
were obtained
by  Dyn and Ron \cite{DyRo}, 
Bejancu \cite{BejancuBoundary},
Johnson \cite{JohnsonOvercomingBoundary},
Schaback \cite{SchabackDoubling},
and  
DeVore and Ron \cite{DeRo} -- these show for
various schemes that the approximation order $2m$ can be attained when
the boundary can be neglected (by considering centers that are reasonably sampled throughout $\reals^d$, or 
in a sufficiently large  neighborhood of $\Omega$,  or by considering functions which are compactly supported in $\Omega$ 
or come from
some other (smaller) class of functions for which boundary effects are not an issue). 
This approximation order is illustrated in Figure \ref{Fi:orders}
as a solid, horizontal line.

\begin{center}
\begin{figure}[h]
\centering
\begin{tikzpicture}
\coordinate (y) at (0,7);
    \coordinate (x) at (9,0);
    \draw[axis] (y) node[above] {$\gamma$} -- (0,0) --  (x) node[right] {$\mathit{1/p}$};

      \coordinate (free) at ($0.9*(y)$);
    \coordinate (ub) at ($0.5*(y)$);
    \coordinate (duchon) at ($0.05*(y)$);
     \coordinate (jleft) at ($0.1*(y)$);
    \coordinate (L2) at ($.45*(x)$);
         \coordinate (jcorner) at ($(L2)+.55*(y)$);
    \coordinate (L1) at ($.9*(x)$);
    \coordinate (corner) at ($(L2)+.5*(y)$);
    \coordinate (ubright) at ($(L1)+.6*(y)$); 
    \draw[important line] let \p1=(free), \p2=(L1) in (\p1) node[left] {$2m$} -- (\x2, \y1);
    \draw[dotted] let \p1=(duchon), \p2 =(corner) in (\p1) node[left] {$m - d/2$} -- (\x2, \y2);
    \draw[dotted] let \p1=(corner), \p2 =(L1) in (\p1) -- (\x2, \y1);
    \draw[dashed] let \p1=(ub), \p2=(ubright) in (\p1) node[left] {$m$} -- (\p2);
    \draw[dashdotted] let \p1=(jleft), \p2 =(jcorner) in (\p1) node[above left] {$m - \frac{d-1}{2}$} -- (\x2, \y2);
    \draw[dashdotted] let \p1=(jcorner), \p2 =(ubright) in (\p1) -- (\p2);
    \draw let \p1=(L2) in (\p1) node[below] {$1/2$};
    \draw let \p1=(L1) in (\p1) node[below] {$1$};
  \end{tikzpicture}
  \caption{Graphs of the boundary-free $L_p$ approximation order (solid), 
  Johnson's upper bound on the $L_p$ approximation order in the presence of the boundary (dashed) 
  and Duchon's $L_p$ approximation order (dots).
  The current best $L_p$ approximation order in the presence of a smooth boundary
  is the dash-dotted broken line. }\label{Fi:orders}
  \end{figure}
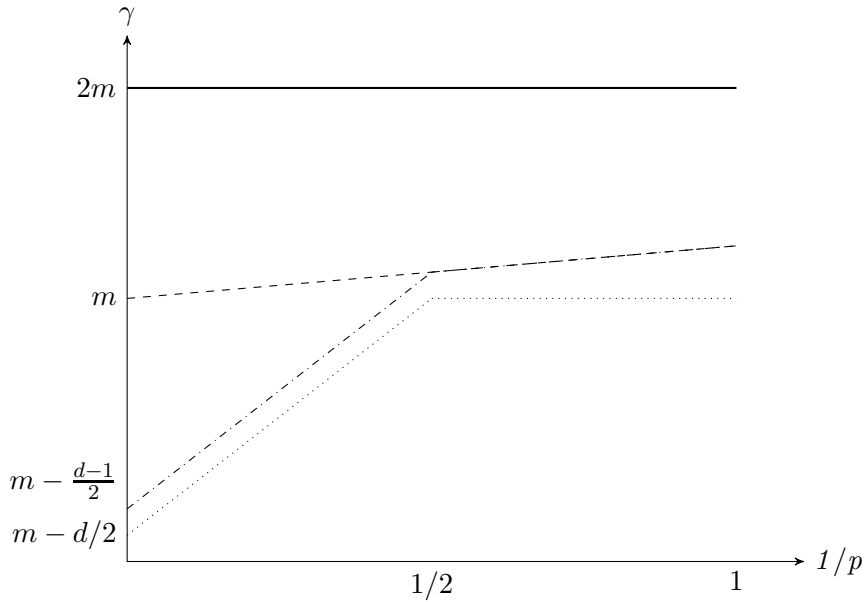
\end{center}  

The inverse result of Johnson \cite{MJ2}, 
shows that for  $\Omega = B $, the unit ball in $\reals^d$, $\Xi \subset (1-\frac12 h)B$ and for any $J$,
$1\le p\le \infty$, there exists $f \in C^{\infty}(\overline{B})$ such that
\begin{equation}\label{JUB}
\mathrm{dist}(f,S_{J}(\Xi,\phi_{m,d}))_p\ne o(h^{m+1/p}).
\end{equation}
(This result holds regardless of the polynomial space $\Pi_{J}$, including $\Pi_{-1} = \{0\}$.)
This upper bound on the approximation order in the presence of a boundary is
illustrated in Figure \ref{Fi:orders}
as a dashed line.

The current state of the art for surface spline approximation with scattered
centers in bounded domains comes from interpolation by functions in
$S_{m-1}(\Xi,\phi_{m,d})$. 
We separate this into two cases, depending on the parameter $p$. 
For $\Omega \subset \reals^d$ having sufficiently smooth boundary and 
for sufficiently smooth $f$ (specifically for $f$ in the Sobolev space
$W_2^{m+1}(\reals^d)$ when $p=1$ and for $f$ in the Besov space
$B_{2,1}^{m+1/p}(\reals^d)$ when $1<p\le2$),
the rate
\begin{equation}\label{MJp<2}
\|f-I_{\Xi}f\|_p = O(h^{m+1/p})
\end{equation} 
holds for $1\le p\le 2$ -- this is to be found in \cite{MJ3}. By the upper bound (\ref{JUB}), this is
the best possible approximation order. On the other hand, in \cite{MJ4} it has been shown that, 
for $p>2$ and for sufficiently smooth $f$ 
(for $f$ in the Besov space $B_{2,1}^{m+1/2}(\reals^d)$), 
\begin{equation}
\|f-I_{\Xi}f\|_p = O(h^{\gamma_p+1/2})\label{MJp>2}
\end{equation}
holds. 
This result, for the case $p=2$, has recently  been studied again \cite{loehndorf} using techniques from elliptic PDEs, and confirming the
saturation order on $B_{2,1}^{m+1/2}$.

Thus, there is a  gap between the
best approximation order for $p>2$ and Johnson's upper bound (\ref{JUB}).
This situation is reflected in Figure \ref{Fi:orders}.
Moreover, the classes of functions for which (\ref{MJp>2}) and (\ref{MJp<2})
hold -- except when $p=2$ -- are smaller than one would expect (in particular, for (\ref{MJp<2}) where $1\le p<2$, smoothness is measured
in the stronger $L_2$ norm, rather than the weaker $L_p$ norm).

In this article, we show that the convergence rate $\dist(f,S_{m-1}(\Xi, \phi))_p = \mathcal{O}(h^{m+1/p})$ holds for  target functions $f\in B_{p,1}^{m+1/p}(\Omega)$ when $1<p<\infty$,
 and slightly smaller spaces when $p=1,\infty$.

%
%
\subsection{Overview}
The goal of this paper is to demonstrate that the representation (\ref{basicrep}) holds, 
to study regularity properties of auxiliary functions $g_j$,
to use this to attack the boundary effects in surface spline approximation with the aid of the scheme (\ref{scheme}),
and to give an explicit representation of the Beppo-Levi extension operator.

The basic strategy of using the solution of (\ref{Dirichlet}) to obtain (\ref{basicrep}) is introduced in Section \ref{S:ID}. 
This section contains the main theorems concerning the solution of (\ref{Dirichlet}), the validity of the identity (\ref{basicrep})
and the regularity of the boundary operators $N_j$ 
(although the more involved proofs are given later).

 Mapping properties of the boundary layer operators 
 used in (\ref{basicrep}) and in the solution of the Dirichlet problem are 
 studied in Section \ref{boundary_layer_potential_operators}. In particular, 
 the regularity of such operators ``up to the boundary''  is studied here,
 as well as jump conditions and transposition of the boundary operators.
 These results may be well known to some readers (many can be found in \cite{Agmon} for instance); they are
 included here to keep the manuscript self-contained and because these results are used in later sections.
 
 Section \ref{Soln} treats the solution of (\ref{Dirichlet}) by a boundary integral method adapted from a technique treating
  the biharmonic problem used in (\cite{CD,ChZh}). 
 It  recasts the problem initially  as an  integral equation which can be solved by providing a bounded inverse
 to an integral operator $L$ acting between reflective Banach spaces.
 
Section \ref{Section psdo} uses the theory of pseudodifferential operators to analyze the problem. It calculates the
(multi-ordered) principal
symbol of the integral operator described in  Section \ref{Soln} and shows that it is elliptic. This is used to determine mapping
properties of $L$, as well as to show that $L$ has closed range.
 
Section \ref{proofs} gives proofs of the main theorems (which have been stated in Section \ref{S:ID}).

In Section \ref{S:SSA} we present and study the surface spline approximation scheme which treats functions defined on bounded regions using
  $S_{m-1}(\Xi, \phi_{m,d})$. 
 The section is devoted to establishing the approximation power of this scheme, and to showing how oversampling near the boundary
 can overcome boundary effects.
 
In Section \ref{S:Ext} we discuss how (\ref{basicrep}) provides the extension which minimizes Sobolev semi-norm.

%
%
\subsection{Notation and background}
\label{notation}
\paragraph {Types of domains considered:} We consider bounded, connected, open $\Omega\subset\reals^d$ 
having  a $C^{\infty}$ outer normal, which we denote by $\vec{n}:\partial \Omega \to \mathbb{S}^{d-1}$.
In a  neighborhood 
$
\N(\partial \Omega) 
:=\partial \Omega+B(0,\epsilon_0)$ of the boundary 
of 
$\Omega$ 
we can describe $\partial \Omega$ as the zero set 
 of a ``signed distance function'' 
$\rho:
\N(\partial \Omega)   
\to (-\epsilon_0, \epsilon_0)
$.
This means that  for all
$
 x\in \N(\partial \Omega), $
there is a unique $ \gamma(x)\in \partial \Omega$
with
$\dist(x,\partial \Omega) =|x-\gamma(x) | = |\rho(x)|$ 
which satisfies $ \rho(x)<0 $ if and only if $   x\in \Omega.$

By extending the normal vector field to the neighborhood of the boundary (writing 
$\vec{n}:\N(\partial \Omega) \to \mathbb{S}^{d-1}$ for the extension) via 
$\vec{n}(x) = \vec{n}(\gamma(x))$, we can smoothly extend
the boundary differential operators (\ref{bdry_operator_def}) to $\N(\partial \Omega)$ as well:
\begin{equation}\label{extended_operator_def}
\Lambda_j f(x):= 
\begin{cases}
 \Delta^{\frac{j}{2}}f(x)&\quad \text{for even $j$,}\\
\sum_{\ell=1}^d \vec{n}_{\ell}(x) \frac{\partial}{\partial x_\ell} \Delta^{\frac{j-1}{2}}f(x) &\quad \text{for odd $j$.}
\end{cases}
\end{equation}
From this it follows that the boundary operators defined in (\ref{bdry_operator_def})  
are simply the composition of
these differential operators with the trace operator:  
$\lambda_j = \tr \Lambda_j$.


\paragraph{Normal and tangential coordinates} 
If $O'\subset \reals^{d-1}$ and $U'\subset \partial \Omega$, with $\tilde{\Psi}:O'\to U'$ a diffeomorphism,
we can generate, for $\epsilon<\epsilon_0$,  tangential and normal coordinates in $U= U'+B(0,\epsilon)\subset \N(\partial \Omega)$
via 
\begin{equation}\label{explicit_normal_tangential}
\Psi:O\to U:\bfx=(x',x_d) \mapsto \tilde{\Psi}(x') + x_d \vec{n}\bigl(\tilde{\Psi}(x')\bigr).
\end{equation}
Here $O = O'+B(0,\epsilon)$.
We define %
smooth vector fields
$\bfe_j(\bfx) = \frac{\partial}{\partial x_j}\Psi (x_1,\dots,x_d)$, for $1\le j\le d$. 
The Gram matrix (of the Jacobian $D\Psi$ of $\Psi$), $\sfg: O\to \mathrm{GL}(d,\reals)$, 
is defined via its entries $\sfg_{i,j} = \langle \bfe_i(\bfx), \bfe_j(\bfx)\rangle$. Its inverse we denote by 
$\sfg^{-1} = \left(\sfg^{i,j}\right)_{i,j}$.

We note that the signed distance $\rho(\Psi(x)) = x_d$, so $(\nabla \rho) (\Psi(\bfx)) = \vec{n}(\tilde{\Psi}(x')) = \bfe_d(\bfx)$.
For fixed $t\in (-\epsilon_0,\epsilon_0)$, 
let $\partial \Omega_t = \{u\in \reals^d \mid \rho(u)  = t\}$.
 When $x_d =t$, 
 the image of $\Psi$ is the level set 
$M_t :=  \partial \Omega_t\cap U = \Psi(O'\times \{t\}).$ 
Since  $(\nabla \rho) (\Psi(\bfx)) $ is normal to $M_t$ at $\Psi(\bfx)$,
we have that $\langle \bfe_j(\bfx), \bfe_d(\bfx)\rangle = \delta_{j,d}$. 
Consequently, $\sfg$ and $\sfg^{-1}$ have a two block structure 
(with a $1\times 1$ block and another  of size $d-1\times d-1$). In short, we write 
$$
\sfg
(\bfx) = \begin{pmatrix} \sfg_{d-1}(\bfx) & 0 \\ 
0& 1\end{pmatrix}\qquad \text{and}\qquad \sfg^{-1}(\bfx) = \begin{pmatrix} 
\bigl(\sfg_{d-1}(\bfx)\bigr)^{-1} & 0 \\ 
0& 1\end{pmatrix}.
$$

\paragraph {Distributions:} 
For an open set $U\subset \reals^d$ we denote the space of test functions supported on compact subsets
of $U$ by $\D(U) = C_{0}^{\infty}(U)$. The space of distributions on $U$ is denoted $\D'(U)$.
Similarly, the space of $C^{\infty}$ functions is $\E(U) = C^{\infty}(U)$ and the space of compactly supported distributions is $\E'(U)$, while
the space of Schwarz functions is denoted $\Ss(\reals^d)$ and the space of tempered distributions is $\Ss'(\reals^d)$.

For  an open set $U\subset \partial \Omega$, 
$\D(U)$ and $\E(U)$ retain the same meaning as do $\D'(U)$ and $\E'(U)$.
Because $\partial \Omega$ is compact,
$\E(\partial \Omega)$ and $\D(\partial \Omega)$ coincide,
as do  $\E'(\partial \Omega)$ and $\D'(\partial \Omega)$.
Because $\partial \Omega$ is endowed with the surface measure $\sigma$ 
there is a natural identification between locally integrable real-valued functions and real distributions via the pairing
$\langle g, \phi \rangle = \int_{\partial \Omega} g(x) \phi(x) \dif \sigma(x)$ 
(valid for all $\phi \in \D(\partial \Omega)$). For an operator on distributions, we use $()^{\t}$ to indicate the transpose with respect to this pairing,
so $\langle M^{\t} T,f\rangle = \langle T,M f \rangle$.

We identify $C^{\infty}(U,\reals^m)$ and 
$C_c^{\infty}(U,\reals^m)$ with $\bigl(\E(U)\bigr)^m$ and $\bigl(\D(U)\bigr)^m$, respectively. 
The duals are
$(\mathcal{D}'(U))^m$ and $(\mathcal{E}'(U) )^m$.

\paragraph{Pullback}
For  open sets 
$U,O\subset \reals^d$ and a smooth diffeomorphism 
$\Psi:O\to U$,  the {\em pullback} 
of a smooth function is $\Psi^*(g)=  g\circ \Psi$.
The pullback extends continuously as a map between $\E'(U) \to \E'(O)$  and $\D'(U) \to \D'(O)$.
See \cite{Horm1}).

The pullback of the surface measure $\delta_{\partial \Omega}:g\mapsto \int_{\partial\Omega} g(x) \dif \sigma(x)$
can be computed by writing $\partial \Omega$ as the zero set of the signed distance function $\rho:\reals^d \to \reals$. 
 We have
 $\rho^* \delta = \delta_{\partial \Omega}$ (cf. \cite[Theorem 6.1.5]{Horm1}).
 If $\Psi:O\to U$ maps $O\cap \reals^{d-1}$ to $U\cap \partial \Omega$ (for instance, if we use tangential and normal coordinates)  
 then it follows that 
 $\Psi^* \delta_{\partial \Omega} = (\rho\circ \Psi)^* \delta = \delta_{\reals^{d-1}}$, the standard Lebesgue measure on $\reals^{d-1}\times \{0\}$.
Distributions of the form $f\cdot \delta_{\partial \Omega} $, 
	supported in $U$ are transformed according to
		$\Psi^* \left(f\cdot \delta_{\partial \Omega}\right) = \left( \Psi^*f \right) \cdot\delta_{\reals^{d-1}} .$

\paragraph{Coordinate change}
By conjugating with $\Psi^*$, we  express
an operator $A:\D'(U)\to \D'(U)$ {\em in coordinates} on $O$ as $A^{\Psi}$.
Thus, we write
$A^{\Psi} = \Psi^* A(\Psi^*)^{-1}$.

For $f\in C^{\infty}(U)$ let $F = f\circ\Psi$. 
Then
$\nabla^{\Psi} F = \sum_{k=1}^d \sum_{j=1}^d \sfg^{jk} \frac{\partial F}{\partial x_j}  \mathbf{e}_k$.
The Laplace operator in coordinates is 
$$  
\Delta^{\Psi}F(x) =
\sum_{j,k=1}^{d} 
\frac{1}{\sqrt{\det{\sfg} }}
\frac{\partial}{\partial x_j} 
\left( \sfg^{k,j}\sqrt{\det{\sfg}}  \frac{\partial }{\partial x_k} 
F (x)
\right). 
$$

\paragraph{Operators in normal and tangential coordinates}
For $\bfu = \Psi(\x)\in \partial \Omega_t$, the unit normal  is $\bfe_d(\bfx)$.
The vector fields $\bfe_1|_{ \partial \Omega_t},\dots,\bfe_{d-1}|_{ \partial \Omega_t}$, which lie tangent to $ \partial \Omega_t$,  
have corresponding Gram matrix  
$\sfg_{d-1}|_{M_t}$.
The Laplace--Beltrami
operator $\Delta_{t}$ for $\partial \Omega_t$ is given in coordinates by
$${\Delta_{t}}^{\Psi} F (\bfx) = 
\sum_{j,k=1}^{d-1} 
\frac{1}{\sqrt{\det{\sfg_{d-1}(\bfx)} }}
\frac{\partial}{\partial x_j} 
\left( \sfg^{k,j}(\bfx)\sqrt{\det{\sfg_{d-1}(\bfx)}}  \frac{\partial }{\partial x_k}F(\bfx)\right).$$
From these observations, it follows that the Laplacian can be decomposed as:
\begin{equation}
\label{Loc_Lap}
\Delta^{\Psi} F
=
{\Delta_{t}}^{\Psi} F
+ 
\frac{\partial^2}{\partial x_d^2} F
+
\mu(\bfx) \frac{\partial}{\partial x_d}F
\end{equation}
 with 
 $
 \mu(\bfx):= \frac{1}{\sqrt{\det \sfg(\bfx)}} \frac{\partial }{\partial x_d}\sqrt{ \det \sfg(x)} =  \Psi^* (\mathrm{div} \vec{n})\in C^{\infty}(O).$

 The extension of the normal derivative operator applied to a distribution $f\in \D'(U)$ obeys
	\begin{equation}\label{normtrans}
		\Psi^* D_{\vec{n}} f = \frac{\partial}{\partial x_d} \Psi_* f .
	\end{equation}
	This follows because 
	$ \Psi^*D_{\vec{n}} f  
	=\langle \vec{n}(\Psi (\cdot)),\nabla^{\Psi} \Psi^*f \rangle
	=\langle \mathbf{e}_d,\nabla^{\Psi} \Psi^*f \rangle
	$.
Likewise, the transpose of the normal derivative 
$D_{\vec{n}}^{\t}f=-\nabla( \vec{n} f) = -\langle \vec{n}, \nabla f\rangle - (\mathrm{div} \vec{n})f$
satisfies
	\begin{equation}\label{normadjointtrans}
		\Psi^* D_{\n}^{\mathrm t} f = -\frac{\partial}{\partial x_d} \Psi^* f - \mu \Psi^*f.
	\end{equation}

\paragraph{Fourier transform} For  $f\in L_1(\reals^d)$, we define $\widehat{f} (\xi) = \int_{\reals^d} f(x) e^{-i\langle x,\xi\rangle} \dif \xi$. This is extended to tempered distributions $\mathcal{S}'(\reals^d)$ in the usual way. For $f \in \mathcal{S}'(\reals^d)$ and $g\in \mathcal{S}(\reals^d)$, we have the usual Plancherel theorem 
$\langle f, g\rangle = \langle \widehat{f}, \widehat g\rangle$.

\paragraph{Smoothness spaces}
For $1\le p<\infty$ and $k\in \nats$, we denote the standard Sobolev space over $\Omega$ by $W_p^k(\Omega)$.  When $p=\infty$, we use
the standard space $C^k(\overline{\Omega})$, of functions having continuous $k$th order derivatives up to the boundary of $\Omega$.
For non-integer orders,
we consider two main extensions.

For $s\in (0,\infty)$, $1\le p<\infty$ and $1\le q\le \infty$, the Besov space $B_{p,q}^s(\Omega)$ is the real interpolation space 
$[W_p^m(\Omega),W_p^k(\Omega)]_{\theta,q}$ with $\theta = \frac{s-k}{m-k}$.  When $p=\infty$ and $s\in (0,\infty)\setminus \nats$, 
we consider  $C^{s}(\overline{\Omega}) $ the H{\"o}lder space; it is well known that 
$C^{s}(\overline{\Omega}) = B_{\infty,\infty}^s(\Omega) = [C^m(\Omega),C^k(\Omega)]_{\theta,\infty}$ where $m,k\in \nats$ and $\theta= \frac{s-k}{m-k}$.
See \cite{Trieb} for background and further references on Besov and H{\"o}lder spaces.

For $1< p<\infty$  and $s\in \reals$ we define the
 the Bessel potential space
$H_p^s(\reals^d)$
as 
$$H_p^s(\reals^d)
:= 
\{f\in \mathcal{S}'(\reals^d)\mid \bigl((1+|\cdot|^2)^{s/2} \widehat{f}\bigr)^{\vee}\in L_p(\reals^d)\}.
$$
It is the pre-image under the Bessel potential $J^s = (1-\Delta)^{s/2}$ of $L_p(\reals^d)$ for
and 
has  norm
$$\|f\|_{H_p^s} := \left\| \bigl((1+|\cdot|^2)^{s/2} \widehat{f}\bigr)^{\vee}\right\|_p = \|J^s f\|_p.$$
When $s=k$ is a non-negative integer, this coincides with the standard Sobolev space $W_p^k(\reals^d)$. 
Furthermore, these are particular examples of Triebel-Lizorkin spaces, namely $H_p^s =F_{p,2}^s$.
See \cite[1.3.2]{Trieb} and references therein for background.

We denote the space of compactly supported distributions in $H_p^s(\reals^d)$ (resp., $W_p^k(\reals^d)$) by $H_{p,c}^s(\reals^d)$ (resp., $W_{p,c}^k(\reals^d)$).
Likewise, $H_{p,loc}^s(\reals^d) = \{f\in \mathcal{D}'(\reals^d)\mid (\forall \psi\in \mathcal{D}(\reals^d)) f\psi \in H_p^s(\reals^d)\}$,
and $W_{p,loc}^s(\reals^d)$ has the obvious modification.

Of special importance is the fact that, for all $s\in \reals$, pointwise multiplication by smooth functions is continuous: for every $s,p$, there is a constant $C$ and an integer $m\in \nats$
so that 
if $f\in H_p^s(\reals^d)$ and $g\in C^{\infty}$ then
$\|fg\|_{H_s^p} \le C \|g\|_{C^m} \|f\|_{H_p^s}$ (see \cite[Theorem 4.2.2]{Trieb}).
Similarly, for a diffeomorphism $\Phi:\reals^d\to\reals^d$, there is a constant $C$ so that for all $f\in H_p^s$ we have
$\|\Phi^*f\|_{H_p^s} \le C \|f\|_{H_p^s}$. 
It follows that if $K\subset O$ is compact and $\Phi:U\to O$ is a diffeomorphism between open sets in $\reals^d$, 
then there is a constant $C_K$ so that for all $f\in H_p^s$ with support
$\supp{f}\subset K$, the estimate 
$\|\Phi^*f\|_{H_p^s} \le  C_K \|f\|_{H_p^s}$ holds.

By duality $(L_p)' \sim L_{p'}$ (for $\frac1p +\frac1{p'} =1$) and the fact that the operators $J^s$ form a 
1-parameter group on $\Ss'$, the dual of  $H_p^s(\reals^d)$ is
identified with $H_{p'}^{-s}(\reals^d)$ in the sense that
for any $\lambda \in \bigl(H_p^s(\reals^d)\bigr)'$ there is a unique distribution $g\in H_{p'}^{-s}(\reals^d)$
for which $\lambda(f)  = \langle g, f\rangle_{H_{p'}^{-s},H_{p}^{s}}$, where the pairing 
 $\langle g, f\rangle_{H_{p'}^{-s},H_{p}^{s}}$ is the extension by continuity of the above bilinear form. This is  \cite[Remark 7.1.9]{Trieb73}

\paragraph{Smoothness spaces on $\partial \Omega$}
Let $(U_{j},\Phi_{j}:U_j\to O_j\subset \reals^{d-1})$ be an atlas for $\partial \Omega$, 
and let $(\tau_{j})$ be a partition of unity
subordinate to $(U_{j})$.
For $1< p< \infty$, we define the Bessel potential spaces 
$H_p^s(\partial \Omega)$ by way of the norm
$$\|f\|_{H_p^s(\partial \Omega)} := \sum_{j} \|(\Phi_{j}^{-1})^* ({\tau}_j f)\|_{H_p^s(\reals^{d-1})}.$$

\paragraph{Pseudodifferential operators:} 
In Sections \ref{Section psdo} 
 and \ref{proofs} 
  we make use of the theory of pseudodifferential operators. Key results will be discussed at the beginning of Section \ref{Section psdo}.
   For background on this topic, we direct the reader to 
\cite{Grubb,Horm3,StRay,Tayl,Trev}.

\section{Multilayer representation of functions}\label{S:ID}
In this section we discuss the key identity
\begin{equation}
\label{Main_Rep}
f (x) = \int_{\Omega} \Delta^m f(\alpha) \phi(x-\alpha)\, \dif\alpha 
+
\sum_{j=0}^{m-1} \int_{\partial \Omega} 
g_j(\alpha)
 \lambda_{j,\alpha} \phi(x-\alpha)\, \dif\sigma(\alpha) +
p(x),\ 
\end{equation}
(with $ p\in \Pi_{m-1}$), 
which we later show is valid for sufficiently smooth functions.  
The identity determines $f$ from its $m$-fold Laplacian and $m$ boundary layer potentials 
$V_jg_j (x):= \int_{\partial \Omega} g_j(\alpha)
 \lambda_{j,\alpha} \phi(x-\alpha)\, \dif\sigma(\alpha)$, 
each involving an auxiliary boundary function $g_j$ and  a kernel $ \lambda_{j,\alpha} \phi(x-\alpha)$
obtained by applying the $j$th boundary operator
 to the function (which depends on the {\em order} $m$ and spatial dimension $d$)
\begin{equation}\label{Surf_Spline_Def}
\phi(x)=\phi_{m,d}(x):=
C_{m,d}\begin{cases}
|x|^{2m-d}\log|x|\quad& d\, \text{is even}\\
|x|^{2m-d}\quad&d\,\text{is odd},
\end{cases}
\end{equation}
which is notable for  being a fundamental solution of $\Delta^m$ in $\reals^d$,
cf. \cite[(2.11)]{Poly}. As such, it is in $C^{\infty}(\reals^d\setminus\{0\})$; in fact,
$\Delta^m\phi (x) = 0$ for $x\ne 0$. The extension of the operators $V_j$ to distributions is discussed
in the next section.

By direct differentiation of (\ref{Surf_Spline_Def}), one easily sees that 
\begin{equation}\label{FS_derivative_estimates}
|D^{\beta} \phi(x)| \le \begin{cases}  C_{m,d,\beta}|x|^{2m-d-|\beta|}(\log|x| +1),&\quad |\beta|\le 2m-d\\
 C_{m,d,\beta}|x|^{2m-d-|\beta|},&\quad |\beta|> 2m-d \end{cases}
\end{equation}
 (see  \cite[Claim 5]{HL} and the subsequent discussion).
 A consequence, used throughout this article, concerns convolution of $\phi$ with compactly supported
distributions that annihilate polynomials (such convolutions are well defined, at least on the complement of the support 
of the distribution).
\begin{lemma}\label{conv_with_FS}
Let $L\ge 2m-d$. For a compactly supported distribution $F$ for which $F\perp \Pi_L$  and for $x\notin \supp{F}$
$$|F*\phi (x)|\le C (1+|x|)^{2m-d-L-1}.$$
Here the constant $C$ depends on $F$.
\end{lemma}

Adopting the terminology of \cite[Chapter 2]{LionsMagenes},
we note that for any positive integer $\nu$, the system of boundary operators $(\lambda_j)_{j=0}^{\nu-1}$ forms a  
{\em Dirichlet system} of order $\nu$ (cf. \cite[Definition 2.1]{LionsMagenes})
and that $\bigl\{\Delta^m, (\lambda_j)_{j=0}^{m-1}\bigr\}$ is self adjoint
(cf. \cite[\S 2.5]{LionsMagenes}). 
Indeed, we have  {\em Green's formula}
\begin{equation}\label{GreensFormula}
\int_{\Omega}v(x) \Delta^m u(x)  - u(x)  \Delta^m  v(x) \dif x =
 \sum_{j=0}^{2m-1} \int_{\partial \Omega} \lambda_{j}u(x) \lambda_{2m-j-1} v(x) \dif \sigma(x)
\end{equation}
which follows directly from the divergence theorem and 
holds for a general class of domains $\Omega$ 
(we will be satisfied by considering bounded domains with smooth boundaries)
and for all functions $u,v$ in $C^{2m}(\overline{\Omega})$. 
A consequence of this is {\em Green's representation} (see  \cite[(2.11)]{Poly}) for smooth functions:
\begin{multline}\label{GreensRep}
\int_{\Omega} \Delta^m f(\alpha) \phi(x-\alpha) \dif\alpha + \sum_{j=0}^{2m-1} (-1)^j \int_{\partial \Omega} (\lambda_{j} f)(\alpha) \; \lambda_{2m-j-1,\alpha}\phi(x-\alpha)\, \dif \sigma(\alpha)\\
=\begin{cases} f(x)&x\in \Omega,\\ 0 & x\in \reals^d\setminus \overline{\Omega}.\end{cases}
\end{multline}
This identity determines $f$ from its $m$-fold Laplacian and $2m$ boundary values $\lambda_j f$, with $j=
(0,\dots,2m-1)$.

Note that (\ref{GreensFormula}) involves twice as many boundary terms as (\ref{Main_Rep}).
It is unsatisfactory for our purposes (i.e., producing an approximation operator using scattered translates of the fundamental solution $\phi$).
Although we could attempt to discretize (\ref{GreensFormula}) to obtain an approximation operator  similar to (\ref{scheme}), 
the higher order derivatives of $\phi$ at the boundary  are too singular, and would cause a degradation in the approximation power of the scheme.

To simplify the problem, we may
decompose $f=f_1+f_2$
 into the solution of a homogeneous PDE with inhomogeneous boundary conditions
and a complementary part which vanishes to high order at the boundary.

 The first part, $f_1$, is  the solution of the polyharmonic Dirichlet problem (\ref{Dirichlet})
with boundary values obtained from $f$. 
The second part, $f_2 = f-f_1$, vanishes to $m^{\text{th}}$ order at the boundary and satisfies 
$\Delta^m f_2 = \Delta^m f$. 
The identity (\ref{Main_Rep}) will follow if we can write the polyharmonic part $f_1$ in the form
\begin{equation}\label{BhRep}
	f_1(x) = \sum_{j=0}^{m-1} \int_{\partial \Omega} 
	g_j(\alpha) 
	\lambda_j\phi(x-\alpha)\, \dif\sigma(\alpha)+ p(x)
\end{equation}
and if we can justify taking higher (up to $2m-1$) order derivatives of $f_1$ at the boundary.  By Green's representation, we apply (\ref{GreensRep}) to obtain
\begin{equation}
	f_2 =
	  \int_{\Omega} \Delta^m f(\alpha) \phi(\cdot -\alpha)\, \dif\alpha
	  	 +
	\sum_{k=0}^{m-1}  (-1)^k
	\int_{\partial \Omega} 
	\lambda_{2m-k-1}f_2 (\alpha) \; \lambda_{k,\alpha} \phi(\cdot-\alpha)
	 \, \dif\sigma(\alpha),\label{RestRep}
\end{equation}
where we have used the fact that  
the lowest order boundary values  of $f_2$ vanish 
(for $j=0\dots m-1$, $\lambda_j f_2 =  \lambda_j f-\lambda_j f_1 = 0$).
To obtain (\ref{Main_Rep}) we show that:
\begin{enumerate}

	\item[[\;A]] solutions of Dirichlet's problem are of the form (\ref{BhRep});
	
	\item[[\;B]]  for sufficiently smooth boundary data, functions of the form (\ref{BhRep}) are smooth near the boundary.
	
\end{enumerate}

Item [A] is the subjection of Sections \ref{boundary_layer_potential_operators} - \ref{proofs}, where we will demonstrate the following theorem.
\begin{theorem}\label{polyharmonic_dirichlet_solution}
 For functions $h_0,\dots, h_{m-1}$ with 
$h_k\in 
C^{\infty}(\partial \Omega)
$, there is a function 
$u$ 
satisfying (\ref{Dirichlet}) and having the form (\ref{BhRep}) 
with 
$g_j\in 
C^{\infty}(\partial \Omega)
$ for each $j$. 
Moreover, for each $j=0,\dots, m-1$
and $1<p<\infty$, we have, for $s\ge 0$
$$\|g_j\|_{H_p^{s+j+1-2m}(\partial \Omega)} \le C_{s,p} \max_{k=0\dots m-1} \|h_k\|_{H_p^{s-k}(\partial \Omega)}.$$
\end{theorem}
\begin{proof} The proof is given in Section \ref{proofs}.\end{proof}
We note, in particular, that if $h_k = \lambda_k f$ for some $f\in W_2^{m}(\Omega)$, then the trace theorem states that
$h_k \in H_2^{m-k-1/2}(\partial \Omega)$. 
A consequence of the above theorem 
is
an extension which gives the weak boundary layer solution of (\ref{Dirichlet}) with such data.
\begin{corollary}\label{Dirichlet_corollary}
Suppose $h_k\in H_2^{m- k - 1/2}(\partial \Omega)$ for $k=0\dots m-1$. Then there exist
$p\in \Pi_{m-1}$ and  $g_j\in H_2^{ j+1/2-m} (\partial \Omega)$  for $j=0,1,\dots m-1$, 
so that $u = \sum V_j g_j + p\in W_{2,loc}^m(\reals^d)$ and   $u$ solves (\ref{Dirichlet}).
\end{corollary}

Item [B] requires understanding the boundary behavior of the layer potential solution (\ref{BhRep}), which will be developed along the way.

%
%
%
Of course, along with the representation (\ref{Main_Rep}), we also expect the auxiliary
boundary functions $g_j$ to be sufficiently regular, determined by operators 
({\em trace operators}) 
applied to $f$ that map appropriate $L_p$ smoothness spaces continuously 
into $L_p$. This is summarized in the main theorem of this section:
\begin{theorem}\label{Main_Rep_Theorem}
For $f\in C^{2m}(\overline{ \Omega} )$, the representation (\ref{Main_Rep}) holds pointwise,
and for $1< p<\infty$, the representation (\ref{Main_Rep}) holds a.e. for $f\in W_p^{2m}(\Omega)$.
The functions $g_j$ are given by linear operators: $g_j = N_j f$.
For $s\ge0$ and $1< p<\infty$,  the operator $N_j :B_{p,1}^{s+2m-j-1+1/p}(\Omega) \to H_p^s(\partial \Omega)$  is  bounded.
\end{theorem}
\begin{proof}
The proof of this theorem is given in Section \ref{proofs}.
\end{proof}
%
%
%
\section{Boundary layer potential operators}\label{boundary_layer_potential_operators}
We now  consider the boundary layer potential operators $V_j$ defined initially on $L_1(\partial \Omega)$ 
\begin{equation}\label{veejay}
V_j g (x)  = \int_{\partial \Omega} g(\alpha) \lambda_{j,\alpha}\phi(x-\alpha) \, \dif \alpha.\end{equation}
In this section we first demonstrate that there is a natural extension of $V_j$ to distributions acting on 
$\partial \Omega$ and producing tempered distributions over $\reals^d$. These distributions have 
singular support in $\partial \Omega$ (meaning they are smooth functions away from $\partial \Omega$).
This is followed by an investigation of  the smoothness of functions $V_j g_j$ near to the boundary. 
 
 %
 %
 \subsection{Boundary layer potentials as convolutions}
 \label{Boundary layer potentials as convolutions}
The boundary layer potential operators 
introduced in (\ref{veejay}) can be viewed as convolutions of derivatives of $\phi$ 
with certain distributions supported on the boundary $\partial \Omega$. 
The distributions in question are the measures
$g\cdot \delta_{\partial \Omega}: \varphi \mapsto \langle g\cdot \delta_{\partial \Omega},\varphi \rangle 
=
\int_{\partial \Omega}\varphi(\alpha) g(\alpha)\dif \sigma(\alpha)$.
This convolution can be rewritten as a convolution of $\phi$ with derivatives of $g\cdot \delta_{\partial\Omega}$; we now briefly discuss this.

For an integrable function defined on $\reals^d$, and the general $j$th boundary layer potential is
$$V_j g =\phi * \bigl[\Lambda_j^{\t} (g\cdot \delta_{\partial \Omega} )\bigr].
$$
where the formally transposed operator $\Lambda_j^{\t}$ is a differential operator of  order $j$. 
Specifically, it is $\Lambda_j^{\t}  =  \Delta^{\frac{j}{2}}$   when $j$ is even, and
$\Lambda_j^{\t}  = 
- \sum_{\ell=1}^d \vec{n}_{\ell}(x) \frac{\partial}{\partial x_\ell} \Delta^{\frac{j-1}{2}} 
+ 
\sum_{|\beta|\le j-1} A_{\beta}(x) D^{\beta}$ for odd $j$.

The 
expression $\phi * \bigl[\Lambda_j^{\t} (g\cdot \delta_{\partial \Omega} )\bigr]$, interpreted as a convolution between 
the tempered distribution  $\phi$ and the compactly supported distribution 
$ \bigl[\Lambda_j^{\t} (g\cdot \delta_{\partial \Omega} )\bigr]$, is thus a tempered distribution as well.
In other words, the operators  $V_j$ produce distributions on $\reals^d$ with singular support 
$\partial \Omega$ that are 
polyharmonic in $\reals^d \backslash \partial \Omega$. In particular, on $\reals^d\setminus \partial \Omega$ they 
can be represented (pointwise) as the $C^{\infty}$ function
  $V_j g (x) = \left\langle 
 \bigl( \Lambda_{j}\bigr)^{\t} \bigl( g\cdot \delta_{\partial \Omega}\bigr),
   \phi(x- \cdot)
 \right\rangle$.

\paragraph{Extension to distributions} 
For a distribution $g\in \mathcal{D}'(\partial \Omega)$, 
let $g\cdot \delta_{\partial \Omega} \in  \mathcal{E}'(\reals^d)$ be the distribution (supported on $\partial \Omega$) 
satisfying $\langle g\cdot \delta_{\partial \Omega}, \varphi \rangle = \langle g, \varphi \left|_{\partial \Omega}\right.\rangle$ 
for $\varphi \in C^{\infty}(\reals^d)$. 
The map $g\mapsto g\cdot  \delta_{\partial \Omega}$ is continuous
from $\mathcal{D}'(\partial \Omega)$ to $\mathcal{E}'(\reals^d)$.

For $g\in\mathcal{D}'(\partial \Omega)$ we  define  
$V_j g $ as a the convolution $V_j g= \phi*\bigl(( \Lambda_{j}\bigr)^{\t}\bigl( g\cdot \delta_{\partial \Omega})\bigr)$.
It follows that the restriction to $\reals^d\setminus \partial \Omega$ is 
$V_j g(x) = \left\langle 
 \bigl( \Lambda_{j}\bigr)^{\t}\bigl( g\cdot \delta_{\partial \Omega}\bigr),
   \phi(x- \cdot)
 \right\rangle$
 and $V_jg \in C^{\infty}(\reals^d\setminus \partial \Omega)$.
The above convolution  is an important representation of the operator $V_j$, but it does not adequately indicate the behavior of $V_j g$ 
near the boundary $\partial \Omega$; this is considered in the next subsection.

\subsection{Boundary regularity}\label{SS:BR}
By (\ref{FS_derivative_estimates}) the kernel $\lambda_{j,\alpha} \phi(x-\alpha)$ is locally integrable on $\partial \Omega$, 
provided that $0\le j\le 2m-2$ 
(this is the case in  the construction of $f_1 = \sum_{j=0}^{m-1} V_j g_j +p $
in  (\ref{BhRep}) -- only nonsingular kernels are used). 
Unfortunately, this only guarantees a limited smoothness
near the boundary $\partial \Omega$; by dominated convergence,
 for each $j=0,\dots , 2m-2$, $V_j g \in C^{2m - j-2}(\reals^d)$.
This is enough to guarantee the existence of the boundary 
values $\lambda_k f_1$  for $k=0\dots m-1$ required by (\ref{Dirichlet}), 
but it is insufficient for higher derivatives -- for instance, those required by (\ref{RestRep}).

%
%
%
%
 \paragraph{Smoothness up to the boundary} The smoothness of boundary layer potentials in the vicinity of the boundary has been treated in different
forms under the heading of ``transmission conditions'' (cf. \cite{BdM}), and earlier (cf. \cite{Agmon}). 
We follow the approach of Duduchava \cite{Duduchava}, by manipulating 
Green's representation, to get the following result, which illustrates that for smooth $g$, 
boundary layer potentials
$V_j g$ 
have smoothness at the boundary
-- this is a topic we return to in the next section, 
where we consider the mapping properties of operators $\tr \Lambda_k V_j g$.

\begin{lemma}\label{dud_extension}
For an integer $0\le j \le 2m-1$,
let $s$ be an integer greater than $j+1$.  
For
$g\in C^{s}(\partial\Omega)$, there is a function $F\in C^{2m+s-j-1}(\reals^d)$,
so that $\lambda_k F = 0$ for $k =0\dots 2m+s-j-1$, $k\neq 2m-j-1$ and $\lambda_{2m-j-1}F = g$. 
Furthermore,  there is a constant $C$ (independent of $g$) so that
  $\|F\|_{C^{s+2m-j-1}(\reals^d)} \le C \|g\|_{C^s(\partial \Omega)}$.
\end{lemma}
\begin{proof}
Let $L= 2m+s-j-1$. 
Consider the sequence  of $L+1$ ``boundary values'', 
$\mathbf{r} = (r_0,\dots,r_{L}) = (0,\dots,0, g, 0,\dots,0)$ 
where each entry is zero except the $2m-j-1$ entry, which is $g$.
We construct $F$ as follows. 
After considering a partition of unity $(\tau_i)$ subordinate to $(U_i)$, we work
locally with normal/tangential coordinates given by $\Psi_i:O_i\to U_i$ as described in Section \ref{notation}.
On each $O_i$, we  extend the pullbacks $\Psi_i^* r_J:O_i' \to \reals$ to a function $f_i:O\to \reals$; $F = \sum \tau_i (\Psi^*)^{-1} f$

The transformed boundary operators  on $O$ are denoted $(\Lambda_j)^{\Psi}$ for $j=0,1,2,\dots $. 
The fact that $\bigl((\Lambda_j)^{\Psi}\bigr)_{j=0}^{L-1}$ is a Dirichlet system allows us to obtain new boundary values 
$$f_J 
:= 
\frac{d^J}{dx_d^J} \Psi^* F
= 
\sum_{k\le J} T_{J,k}(\Lambda_k)^{\Psi} \Psi^* F
=
\sum_{k\le J} T_{J,k}\Psi^*r_k,
$$ 
where $T_{J,k}$ is a differential operator of order $J-k$
on $O'$.
This is  \cite[Lemma 2.3] {LionsMagenes}. 

We produce the full collection of ``jets''  of order $L$ along $O'$ by defining, for $0\le J\le L$, and $\alpha'\in \ints_+^{d-1}$,
$u_{\alpha',J} (x',0):= D^{\alpha'}f_J(x',0)$. These satisfy the requirement of Whitney's extension theorem over $O'$ given in 
 \cite[Theorem 2.3.6]{Horm1}, so there is an extension $f:O\to \reals$ in $C^{2m+s-j-1}(O)$ 
 satisfying $D^{\alpha} f(x) = u_{\alpha}$ and $\|f\|_{C^{L}(\reals^d)}\le C \max \|f_J\|_{\infty}$.
 
Because $(\Lambda_J)^{\Psi} =  \sum_{k\le J} \tilde{T}_{J,k} \frac{d^k}{dx_d^k}  $ with each $\tilde{T}_{J,k}$ 
a differential operator of order $J-k$
(again by  \cite[Lemma 2.3] {LionsMagenes}),
we have  $(\Lambda_J)^{\Psi} f = \Psi^* r_J$ and 
$\|f\|_{C^{L}(\reals^d)} \le C \max \|\Psi_J^* r_J\|_{C^J}$. 

\end{proof}

%
%
\begin{lemma}\label{dud_jump}
For integers $j,s$, with $0\le j \le 2m-1$ and $s>j+1$, let  
$g\in C^{s}(\partial\Omega)$ and let $F\in C^{2m+s-j-1}(\reals^d)$ be the function guaranteed by Lemma \ref{dud_extension}. 
Then there is $G\in C^{s-j-1}(\reals^d)$,
so that 
$$
V_j g (x)= (-1)^{j-1}
\begin{cases}\phi*G(x)  - F(x),& x\in\Omega \\ 
 \phi*G(x),& x\in \reals^d\setminus \overline{\Omega}
 \end{cases}
$$
\end{lemma}
\begin{proof}
  By applying Green's representation  (\ref{GreensRep}), we see that,
  \begin{multline*}
  \int_{\Omega} \Delta^m F(\alpha) \phi(x-\alpha) \dif \alpha 
  + 
  (-1)^{j} \int_{\partial \Omega} g(\alpha) \lambda_{j,\alpha}\phi(x - \alpha) \dif \sigma(\alpha)\\
  =
  \begin{cases} F(x) &x\in \Omega\\0&x\in \reals^d \setminus \overline{\Omega}.\end{cases}
  \end{multline*}
Let   
$G = \chi_{\Omega} \Delta^m F$
 denote the extension by zero of $\Delta^m F$ outside of $\Omega$.
With $L= 2m+s-j-1$, 
it follows that $G\in C^{L-2m}(\reals^{d})$, since
 $\Delta^m F \in C^{L-2m}(\reals^d)$ and 
  $ \lambda_{2m+k} F = 0$ for $k=0,\dots, L-2m$.
  \end{proof}
\begin{corollary}\label{boundary_smoothness}  
For $j \in \nats$,
let $s$ be an integer greater than $j+1$.  
For
$g\in C^{s}(\partial\Omega)$
the boundary layer potential 
$V_j g =\int_{\partial \Omega} g(\alpha) \lambda_{j,\alpha}\phi(\cdot - \alpha) \dif \sigma(\alpha) \in 
C^{s+2m-j-2}(\overline{\Omega})$
as well as in
$C^{s+2m-j-2}(\reals^d\setminus \Omega)$. 
Furthermore,
$$\|V_j g\|_{C^{s+2m-j-2}(\overline{\Omega})} \le C \|g\|_{C^s(\partial \Omega)}$$ as well as
$$\|V_j g\|_{C^{s+2m-j-2}(K\cap \reals^d\setminus\Omega)} \le C_K \|g\|_{C^s(\partial \Omega)}$$
for each compact $K\subset \reals^d$.
\end{corollary}
%
\begin{proof}
  Since both $G * \phi $ and $F -G * \phi$ are in $ C^{L-1}(\reals^d)$, the proposition follows
  in case $j<2m$.
  
  For general $j\in \nats$, we simply observe that for $j = 2mr +j'$ (with $0\le j'<2m$), the identity
  $V_j g = \int_{\partial \Omega} g(\alpha) \lambda_{j,\alpha}\phi(\cdot - \alpha) \dif \sigma(\alpha)
  =\Delta^{rm} V_{j'} g$ is valid for $x\notin \partial \Omega$. 
  The result follows
  because $V_{j'} g$ is in $C^{s+2m-j'-2}(\overline{\Omega})$ (resp., is in  $C^{s+2m-j-2}(\reals^d \setminus{\Omega})$) 
  and therefore $\Delta^{rm} V_{j'} g$ is in $C^{s+2m-j-2}(\overline{\Omega})$ 
  (resp. $C^{s+2m-j-2}(\reals^d \setminus{\Omega})$) .

\end{proof}
%
%
%
At this point,  we note that  increased smoothness (beyond $C^{2m-j-2}$) of $V_j g$ cannot be extended across the boundary.
Indeed Lemma \ref{dud_jump} gives the following classical jump conditions.
\begin{corollary}\label{jump_condition}
For integers $j,s$ with $0\le j \le 2m-1$ and $s>j+1$, let  
$g\in C^{s}(\partial\Omega)$.
Then for $k =0\dots 2m+s-j-2$, $k\neq 2m-j-1$, we have for $x\in \partial \Omega$
$$
\lim_{\substack{y\to x\\ y\in \Omega}} \Lambda_k V_j g (y)=
\lim_{\substack{y\to x\\ y\in \reals^d\setminus\overline{\Omega}}}\Lambda_k V_j g(y)
$$
while for $k = 2m-j-1$, we have
$$\lim_{\substack{y\to x\\ y\in \Omega}} \Lambda_k V_j g (y) -
\lim_{\substack{y\to x\\ y\in \reals^d\setminus\overline{\Omega}}}\Lambda_k V_j g(y) = (-1)^j g(x).$$
\end{corollary}

We will return to these jump discontinuities in Section \ref{Uniqueness}.

\subsection{Boundary operators}\label{SS:BOs}
Note that Corollary \ref{boundary_smoothness} implies that 
$V_j:C^{\infty}(\partial \Omega) \to C^{\infty} (\overline{\Omega})$ is continuous 
(with the usual Fr{\'e}chet space topologies on 
$C^{\infty}(\partial \Omega)$ and $C^{\infty} (\overline{\Omega})$).
This permits us to define the following operators, which we call ``boundary operators''. 

\begin{definition}For  $j, k\in \nats$, let
$v_{k,j}^+:C^{\infty}(\partial \Omega)\to C^{\infty}(\partial \Omega)$
be the operator defined for $g\in C^{\infty}(\partial \Omega)$ as
$$v_{k,j}^+ g(x):=\lim_{y \in \reals\setminus \overline{\Omega}\to x} \Lambda_k V_j g(y).$$
Likewise, let $v_{k,j}^-:C^{\infty}(\partial \Omega)\to C^{\infty}(\partial \Omega)$
be  defined  as
$$v_{k,j}^- g(x):= \lim_{y \in \Omega \to x} \Lambda_k V_j g(y).$$
\end{definition}

\begin{remark}\label{self_transpose}
By the local integrability of $\Lambda_{k,x} \Lambda_{j,\alpha}\phi(x-\alpha)$ when 
$k+j\le 2m-2$, it follows that we can take 
$v_{k,j}^+ = v_{k,j}^-=\mathrm{Tr}(\Lambda_k V_j \left|_{C^{\infty}(\partial \Omega)}\right.)$. 
In this case, we drop the $\pm$ notation and simply write $v_{k,j}$.

We note also that when $k+j\le 2m-2$, then $v_{k,j}^{\t} = v_{j,k}$. 
This follows because $\phi$ is even, 
so $\Lambda_{k,\alpha} \phi(x-\alpha)  =\Lambda_{k,\alpha} \phi(\alpha- x)$ for all $k$.
Hence 
$\int_{\partial \Omega} s(x) v_{k,j} g(x) \dif \sigma(x) = \int_{\partial \Omega}  g(x)v_{j,k} s(x)  \dif \sigma(x)$
follows, with exchange of limits justified by the local integrability of the kernel $\lambda_{k,x}\lambda_{j,\alpha} \phi(x-\alpha)$.

In contrast to the case $j+k\le 2m-2$,  
we note that we have for $k+j=2m-1$, $v_{k,j}^- \ne v_{k,j}^+$.
Indeed,
 $v_{k,j}^- g - v_{k,j}^+ g = (-1)^jg$, by the observation in Corollary \ref{jump_condition}.
\end{remark}

In subsequent sections we will 
express the boundary operators $v_{k,j}^{\pm}$ 
as pseudodifferential operators. The 
symbol classes to which they belong
(determined in Section \ref{BOLC}) resolve their regularity. 
%
%
%
\begin{lemma}\label{smoothing_lemma}
Let $1<p<\infty$ and take $s\in \reals$. 
Then for $j,k\in \nats$,
the operators $v_{k,j}^+$ and $v_{k,j}^-$ are bounded from $H_p^s(\partial \Omega)$ to 
$H_p^{s+2m-1 -j-k}(\partial \Omega)$.
\end{lemma}  
\begin{proof}
The proof is postponed until Section \ref{BOLC}.
\end{proof}
%
%
%
\section{The Solution of the Dirichlet Problem}\label{Soln}
We now focus  on solving  the polyharmonic Dirichlet problem (\ref{Dirichlet}) using boundary layer potentials.
To this end, we follow the approach taken by Chen and Zhou \cite{ChZh}[Chapter 8], 
with our main points of departure being that we consider 
the Dirichlet problems in higher dimensions (i.e., $d\ge 2$),  for 
higher order polyharmonic equations (i.e., $m\ge 2$) 
and
for boundary data from Sobolev spaces 
$H_p^s\times H_p^{s-1}\times\cdots\times H_p^{s-m+1}$ with $p$ in the range $1<p<\infty$
rather than for data from $L_2$ Sobolev spaces 
$H_2^s\times H_2^{s-1}$. (Many of these changes are modest, if technical. 
However, the change to higher order $m$ requires greater care in 
demonstrating ellipticity of the system -- this is considered in Section \ref{parametrix}).

We may seek a function of the form 
$T{\bfg} :=  \sum_{j=0}^{m-1} V_j g_j  :\reals^d \to \reals
.$
with $\bfg = (g_j)_{j=0}^{m-1}$. Such an operator can be expressed as a convolution of $\phi$ with
a distribution supported on $\partial \Omega$, i.e., $T\bfg = \phi * \mu_{\bfg}$ as was
discussed in Section \ref{Boundary layer potentials as convolutions}. 
In particular, it solves $\Delta^m T\bfg= 0$ in $\Omega$.
Thus, we simply require  $T\bfg$ to satisfy the boundary conditions, which yields the system of integral equations
\begin{equation}\label{System}
h_{k} = \lambda_{k} \sum_{j=0}^{m-1} V_j g_j = \sum_{j=0}^{m-1}v_{k,j} g_j
\quad \mathrm{ for }\  k =0...m-1
 \end{equation}
 where the operators $v_{k,j} := \lambda_{k} V_j$ have been introduced in the previous section.

Unfortunately, this system is not invertible, in general. 
To treat this, we modify the system by augmenting it with
certain polynomial side conditions. This is explained in the following subsection.
Our goal is to solve this augmented system, 
and we do so in stages. First, we develop the problem further, so that it becomes a problem of inverting an operator on a product of reflexive Sobolev spaces. 
Then we show that this operator possesses an inverse of a sort: a parametrix. Finally, we use the parametrix to prove that a slightly modified version of 
system of integral equations (\ref{System}) is invertible.

\subsection{The System of Integral Equations, Some of the Operators Involved and the Sobolev Spaces Used}
\label{IntEqns}
As a starting point, we consider the system of integral equations, 
$$
\LL
\begin{pmatrix}
  g_0\\ 
  g_1\\
  \vdots\\
  g_{m-1}
\end{pmatrix}
:=
\begin{pmatrix}
    v_{0,0} g_0 + v_{0,1} g_1 + \dots +v_{0,m-1} g_{m-1}\\
    v_{1,0} g_0 + v_{1,1} g_1 + \dots + v_{1,m-1} g_{m-1}\\
   \vdots\\
    v_{m-1,0} g_0 + v_{m-1,1} g_1 + \dots +v_{m-1,m-1} g_{m-1}\\
\end{pmatrix}
=
\begin{pmatrix}
h_0\\
h_1\\
\vdots\\
h_{m-1}
\end{pmatrix}.
$$
By the discussion in Section \ref{SS:BOs},
namely Lemma \ref{smoothing_lemma},
 $\LL$ is continuous 
from 
$(\mathcal{D}'(\partial \Omega))^m\to (\mathcal{D}'(\partial \Omega))^m$. By 
Remark \ref{self_transpose} it is self-transpose.
We look for solutions of the modified system 
$$
\LLs
\left(
\begin{array}{c}
  A_1\\
  A_2\\
  \vdots\\
  A_N\\
\hline
  g_0\\ 
  \vdots\\
  g_{m-1}
\end{array}
\right)
:=
\left(
\begin{array}{c|c}
  0 & 
  P^{\t}\\ 
  \hline
  P
  &  \LL
\end{array}
\right)
\left(
\begin{array}{c}
  A_1\\
  A_2\\
  \vdots\\
  A_N\\
  \hline
  g_0\\
  \vdots\\ 
  g_{m-1}
\end{array}
\right)
=
\left(
\begin{array}{c}
B_1\\
B_2\\
\vdots\\
B_N\\
\hline
h_0\\
\vdots\\
h_{m-1}
\end{array}
\right).
$$
Here $N = \frac{(m-1+d)!}{(m-1)!d!} = \dim(\Pi_{m-1})$, 
and $P:\reals^N \to C^{\infty}(\partial\Omega,\reals^m)$ 
is the Vandermonde-style matrix whose $j^{\text{th}}$ column consists of the 
basic boundary operators 
applied to the $j^{\text{th}}$ basis element for 
$\Pi_{m-1}$, namely $(P)_{kj} = \lambda_k p_j$. Thus,
$$P \begin{pmatrix}
  A_1\\
  \vdots\\
  A_N
  \end{pmatrix} 
  :=
  \begin{pmatrix} A_1 \lambda_0p_1+\dots+A_N \lambda_0 p_N\\
  \vdots\\
  A_1 \lambda_{m-1} p_1 + \dots + A_N \lambda_{m-1} p _N
  \end{pmatrix}.
$$
The operator $P^{\t}:\bigl(\mathcal{D}'(\partial \Omega) \bigr)^m\to \reals^N$ is its natural transpose,
$$P^{\t} \begin{pmatrix} g_0\\ g_1\\ \vdots \\g_{m-1}\end{pmatrix} = \begin{pmatrix}
  \langle g_0,\lambda_0p_1 \rangle + \langle g_1,\lambda_1p_1 \rangle+ \dots +\langle g_{m-1},\lambda_{m-1} p_1 \rangle \\
  \langle g_0,\lambda_0p_2 \rangle   +   \langle g_1,\lambda_1 p_2 \rangle  +\dots+  \langle g_{m-1}, \lambda_{m-1} p_2 \rangle  \\
 \vdots \\
  \langle g_0,\lambda_0p_N \rangle   + \langle g_1,\lambda_1 p_N \rangle  +\dots +   \langle g_{m-1},\lambda_{m-1} p_N \rangle  
  \end{pmatrix}.
$$
 The function $A_1 p_1+A_2 p_2+\dots +A_N p_N+V_0 g_0 +\dots+V_{m-1} g_{m-1}$ solves the 
 Dirichlet problem with $N$ extra ``side conditions''. 
 The relevance of these extra conditions will be made clear in Section \ref{bounded_invertibility_of_L}. 

We restrict $\LLs$ to various products of Bessel potential spaces and recast the problem in the context of  reflexive Banach spaces. 
Thus we make the following definition.
\begin{definition}\label{spaces}
For $1<p<\infty$ and $s\in \reals$, let
\begin{eqnarray}
        X_{p,s}&:=& \prod_{j=0}^{m-1}H_p^{s+j}(\partial \Omega) \\
        X^{\sharp}_{p,s}&:=&\reals^N \times  X_{p,s}
\end{eqnarray}
Similarly, let
\begin{eqnarray}
	Y_{p,s}&:=& \prod_{j=0}^{m-1}H_p^{s-j}(\partial \Omega)\\
       Y^{\sharp}_{p,s}&:=&\reals^N \times Y_{p,s}
\end{eqnarray}
\end{definition}
\begin{remark} 
We have defined the Bessel potential space $H_p^{s}(\partial \Omega)$ in Section \ref{notation}.
We remark that these are smoothness spaces over the manifold $\partial \Omega$ which, 
are reflexive, the dual of 
$H_p^{s}(\partial \Omega)$  
being 
$H_{p'}^{-s}(\partial \Omega)$ 
under the bilinear form 
$ H_p^{s}(\partial \Omega)\times H_{p'}^{-s}(\partial \Omega) \to \comps:(g,h)\mapsto \langle g,h\rangle$
 inherited from the pairing $\langle \phi, \psi \rangle  = \int_{\partial \Omega} \phi(x) \psi(x) \dif x$ defined on 
test functions.
 
From this, we naturally identify the dual of $X_{p,s}$ with $Y_{p', -s}$ and vice versa. 
We have the identification $X_{p,s}^{\sharp}$ with $Y_{p', -s}^{\sharp}$ via the pairing
$$\langle(\vec{A},g),(\vec{B},h)\rangle = \langle g,h\rangle + \sum_{j=1}^{N} A_j B_j.$$ 
\end{remark}

\subsection{Bounded invertibility of $\LLs$}\label{bounded_invertibility_of_L}
The problem we now face is to show that the restriction of $\LLs$ to $X_{p,s}^{\sharp}$ maps onto 
$Y_{p,s+2m-1}^{\sharp}$, and that this map is boundedly invertible. To do this, we make use of the following three 
lemmas, which are proved in the coming subsections.

The first lemma concerns regularity of the operator $L$. It is a direct consequence of the mapping properties
of the constituent pseudodifferential operators $v_{k,j}$, and follows in a more-or-less immediate way
from Lemma \ref{smoothing_lemma}, which in turn, shows boundedness of $L^{\sharp}$ from $X_{p,s}^{\sharp} $ to $Y_{p,s}^{\sharp}$.
%
%
\begin{lemma}\label{psdo} For $1<p<\infty$ and $s\in \reals$, $\LLs$ 
is a bounded map from $X_{p,s}^{\sharp}$ to $Y_{p,s+2m-1}^{\sharp}$. 
\end{lemma}
\begin{proof} 
This follows directly from Lemma \ref{smoothing_lemma}.
\end{proof}
%

The second lemma concerns the range of the map $L_{p,s} := \LLs\left|_{X_{p,s}^{\sharp}}\right.: X_{p,s}^{\sharp} \to Y_{p,s+2m-1}^{\sharp}$.
The following section will demonstrate that there is a near right inverse $R:Y_{p,s+2m-1}\to X_{p,s}$ 
(known as a {\em parametrix})
so that $\LL R = \mathrm{Id}+K$, where $K: Y_{p,s+2m-1} \to Y_{p,s+2m-1}$ is a compact operator. 
%
%
\begin{lemma}\label{closed}  
For $1<p<\infty$ and $s\in \reals$,
$\LLs(X_{p,s}^{\sharp})$ is closed in $Y_{p,s+2m-1}^{\sharp}$.
\end{lemma}
\begin{proof} Postponed until Section \ref{parametrix}.
\end{proof}
%
We remark that the theory of pseudodifferential operators on manifolds can be called upon to prove both Lemmas \ref{psdo} and 
\ref{closed}. 
Specifically, Lemma \ref{psdo} follows by showing that 
$\LL$ is a pseudodifferential operator of a certain order
and Lemma \ref{closed} follows by showing it is elliptic (elliptic pseudodifferential operators are Fredholm operators). 
This is essentially the approach we take in the coming subsections.

The third lemma shows injectivity of the operator $\LLs$. 
This is the moment where using $\LL$ is insufficient,
and the auxiliary polynomial operators $P$ and $P^{\t}$ must be used. 
Essentially, the injectivity of $\LL$ fails to hold because
the polyharmonic Dirichlet problem does not have unique solutions in the unbounded region $\reals^d\setminus \Omega$.
%
%
\begin{lemma}\label{oneone} 
$L_{p,s}$ is $1\!-\!1$. 
\end{lemma}
\begin{proof} 
Postponed until Section \ref{Uniqueness}.
\end{proof}
%

Together, the previous three lemmata imply the following result, which is the key to solving the polyharmonic Dirichlet problem,
and consequently to obtain the desired integral representation.
%
%
\begin{proposition}\label{bounded_inverse}
For $1<p<\infty$ and $s\in \reals$, the map $L_{p,s}=\LLs \left|_{X_{p,s}^{\sharp}}\right.$ is boundedly invertible from $X_{p,s}^{\sharp}$ 
to $Y_{p,s+2m-1}^{\sharp}$.
\end{proposition}
\begin{proof}
It suffices to show that $L_{p,s}^{\sharp}$ is invertible from  $X_{p,s}^{\sharp}$ to $Y_{p,s+2m-1}^{\sharp}$;
the open mapping theorem then guarantees boundedness of the inverse.

Lemma \ref{psdo}, in conjunction with the definition of $\LLs$ and
the duality of the spaces $X_{p,s}^{\sharp}$ and $Y_{p, 1-(s+2m)}^{\sharp}$,
 indicates that $(L_{p,s})^{\t} = L_{p',1-(s+2m)}$. 
By Lemma \ref{oneone}, this operator is $1\!-\!1$ and we have that $\ker (L_{p,s})^{\t} = \{0\}$. 
Since the range of $L_{p,s}$ is closed, we have that 
$$
\mathrm{ran} (L_{p,s}) 
=
\overline{\mathrm{ran} (L_{p,s}) }
=  
\left(\ker (L_{p,s}^{\t})\right)_\perp 
=
\left(\ker (L_{p',1-(s+2m)} )\right)_\perp 
 = 
 Y_{p,s+2m-1}.
 $$
Consequently, $L_{p,s}$ is invertible.
\end{proof}
%

%
%
\section{ Expressing Boundary Layer Potential Operators as Pseudodifferential Operators}
\label{Section psdo}
We continue our investigation of the boundary layer potential operators $V_j$, their boundary values 
$v_{k,j} = \lambda_{k} V_j$, and the full boundary integral operator $L$.
By changing variables 
so that  portions of the boundary $U\cap \partial \Omega$ are flattened,
we may express these as pseudodifferential operators. 
In particular, we can calculate the principal symbols of the boundary operators $v_{k,j}$,
the orders of which
(determined by the order of the principal symbol) 
determine their mapping properties, from which Lemma \ref{psdo} follows naturally.
We  use this calculation to demonstrate the ellipticity of $L$, which guarantees that it has closed range.

\subsection{\bf Background}
Before discussing pseudodifferential operators, we 
mention some other classes of operators which are useful for the theory.
A continuous linear operator  $K:\mathcal{E}'(\reals^d) \to \E(\reals^d)$ is  called a {\em smoothing} operator 
(other texts call it a regularizing or negligible operator). An operator $A:\D(U)\to \E'(U)$ is
{\em properly supported}  if also $A^{\t}:\D(U) \to \E'(U)$; 
by duality, it is clear that such an operator
is continuous also from $\E(U)$ to $\D'(U)$.

\subsubsection{Pseudo-differential operators on Euclidean domains}
Let us now briefly highlight some aspects of the theory of pseudodifferential operators -- these can be found
in a variety of sources (including \cite{Grubb,Horm3,StRay,Tayl,Trev}; this is but a small sampling of resources).

\begin{definition}\label{D:symbol}
Given an open subset $O$ of $\reals^d$ %
for $p\in C^{\infty}(O\times \reals^d)$ %
we say that $p$ is in the symbol class $S_{1,0}^N(X)$ if for each pair of multi-integers 
$\alpha,\beta$ and each compact $K\subset U$ there is a constant $C_{\alpha,\beta,K}$ so that
$$|D_x^{\beta} D_{\xi}^{\alpha}(p(x,\xi))|\le C_{\alpha,\beta,K} (1+|\xi|)^{N-|\alpha|}$$
holds for all $x\in K$. 
\end{definition}

Given a symbol $p\in S_{1,0}^N(O)$, we can (initially) define an operator on test functions as
$$\Op(p){f}(x) := (2\pi)^{-d}\int_{\reals^d}e^{i x\cdot \xi} p(x,\xi) \widehat{{f}}(\xi) \dif \xi.$$

We have 
$\Op(p): \bigl(\mathcal{D}(O)\bigr) \mapsto \bigl(\mathcal{D}'(O)\bigr)$, 
although
the operator can be continuously extended to map 
$\mathcal{D}(O) \to \E(O)$ and $\mathcal{E}'(O) \to \mathcal{D}'(O)$ (cf. \cite[Theorem 1.5]{Tayl}). 
The operator (thus extended) is  a pseudodifferential operator of order $N$. 

Clearly, $S_{1,0}^N(O)\subset S_{1,0}^{N+1}(O)$  
and (because of mapping properties described in item 5 below)  the symbol class
$S^{-\infty}(U) =\bigcap_{N\in\ints}S_{1,0}^N(O)$  consists of operators   
which are smoothing.
For any $p\in S_{1,0}^N(O)$ there is a properly supported operator $P$ and a smoothing  operator 
$R = \Op(r)$, with $r\in S^{-\infty}(O)$  so that $\mathrm{Op}(p)  = P + R$ (cf., \cite[Proposition 7.8]{Grubb}).

We note that this class of operators includes linear partial differential operators: the operator 
$\sum_{\alpha=0}^N a_{\alpha} D^{\alpha}$ has symbol 
$p(x,\xi) = \sum_{|\alpha|=0}^N i^{|\alpha|} a_{\alpha}(x) \xi^{\alpha}$ 
(one checks readily that
$p\in S^N(\reals^d)$).
Another example is $J_s = (1-\Delta)^{s/2}$, which has symbol $(1+|\xi|^2)^{s/2}\in S^{s}(\reals^d)$ 
(this holds for all $s\in \reals$).

\begin{result}\label{results}
The following results hold for such operators:
\begin{enumerate}
\item  
Let $f$ be a distribution on $O$ and $P$ be a pseudodifferential operator. 
If $\Upsilon$ is the largest open set
on which $f$ is smooth (i.e., the complement of the singular support), $Pf$ is $C^{\infty}( \Upsilon)$ as well. 
Indeed, if $\tau$ and $\omega$ are smooth functions on $O$ for which $\supp{\tau}\subsetneq \{x\in O\mid \omega(x)=0\}$, 
then $f\mapsto \omega P(\tau f)$ is smoothing.
\item 
Given a countable sequence of symbols $(p_j)_{j=0}^{\infty}$
with each $p_j\in S_{1,0}^{N_j}(O)$ and $N_j$ decreasing,  there is a symbol 
$p\in S_{1,0}^{N_0}(O)$ such that for every $M\in\nats$,
$ p-\sum_{j=0}^M p_j \in S_{1,0}^{N_{M +1}}(O).$ 
In this case, we write $p = \sum_{j=0}^{\infty} p_j$. 
\item 
For symbols $a\in S_{1,0}^M(O)$ and $b\in S_{1,0}^N(O)$ 
for which one of $\Op(a)$ and $\Op(b)$ is properly supported, 
the composition $\Op(a) \Op(b)$ is a pseudodifferential operator of order $N+M$ and has symbol
$$
(a\odot b)(x,\xi)
= 
\sum_{|\alpha|=0}^{\infty}\frac{(-i)^{|\alpha|}}{\alpha !}D_{\xi}^{\alpha}a(x,\xi)D_x^{\alpha}b(x,\xi)$$
(with convergence of the series understood as in item 2). 
\item 
The class of pseudodifferential operators is closed under diffeomorphism, and order is preserved.
Indeed, we have, for $\Phi:U\to O$, and symbol $p\in S_{1,0}^N(U)$, the operator $(\mathrm{Op}(p))^{\Phi}$ 
is a pseudodifferential operator, with symbol $p^{\Phi} \in S_{1,0}^N(O)$ given by
$$
p^{\Phi} (\Phi(x),\xi)= \sum_{\alpha \in \nats^d}\frac{1}{\alpha!} \phi_{\alpha}(x,\xi) D_{\xi}^{\alpha} p(x,(D\Phi)^{\t}\xi)
$$
Here $ \phi_{\alpha}(x,\xi)$ is a polynomial\footnote{Specifically, 
$\phi_{\alpha}(x,\xi) =[D_{y}^{\alpha} e^{i\langle (\Phi(y)-\Phi(x) -(D\Phi(x))(y-x),\xi\rangle }]\left|_{y=x}\right.$.} 
 in $\xi$
 of degree at most $|\alpha|/2$, and $\phi_0 = 1$.
\item 
For a  symbol $p\in S_{1,0}^N(O)$, the operator $\Op(p)$ maps $H_{p,\mathrm{c}}^s(O)$ boundedly to 
$H_{p,\mathrm{loc}}^{s-N}(O)$ for all $s\in \reals$, $1<p<\infty$.
\end{enumerate}
\end{result}
Item 1 is \cite[Ch.2, Theorem 2.1]{Tayl} as well as \cite[Proposition 7.11]{Grubb}. Item 2 is \cite[Ch. 2, Theorem 3.1]{Tayl} as well as \cite[Lemma 7.3]{Grubb}.
Item 3 can be found in \cite[Ch. 2, Section 4]{Tayl} as well as \cite[Theorem 7.13]{Grubb}. 
Item 4  is \cite[Ch. 2, Theorem 5.1]{Tayl} as well as \cite[Theorem 8.1]{Grubb}.
Item 5  follows from \cite[Ch. 11, Theorem 2.1]{Tayl} .

\subsubsection{Polyhomogeneous operators and ellipticity}
A symbol $p\in S_{1,0}^N(O)$ is {\em positively homogeneous of order $N$} if it satisfies, 
for $\lambda\ge 1$ and $|\xi|\ge 1$,
\begin{equation}\label{weak_polyhomogeneity}
p(x,\lambda \xi) = \lambda^{N}p_j(x,\xi).
\end{equation}
A symbol  $p\in S_{1,0}^N(O)$
is called {\em polyhomogeneous} if it 
 has an (asymptotic) expansion  
$p = \sum_{j=0}^{\infty} p_j$, where each $p_j \in S_{1,0}^{N-j}(U)$
and is positively homogeneous of order $N-j$. 
\begin{definition}
Let 
$S^N(U)$ 
denote the set of  polyhomogeneous symbols of order $N$. Furthermore, for $m\in \nats$,
let $S^N(U,m)$ denote the set of matrix valued symbols $p = (p_{j,k})_{j,k}$ where each $p_{j,k}\in S^N(U)$.
\end{definition}
In the case of a matrix valued symbol, $\Op(p)$ is  defined as
$\Op(p)\mathbf{f}(x) = (2\pi)^{-d}\int_{\reals^d}e^{i x\cdot \xi} p(x,\xi) \widehat{\mathbf{f}}(\xi) \dif \xi$,
where  
$\widehat{\mathbf{f}}(\xi) =  [\widehat{f_0}, \dots \widehat{f_{m-1}}]^{\t}$ 
is the entry-wise Fourier transform 
of $\mathbf{f} = [f_0, \dots f_{m-1}]^{\t}$ and
$p(x,\xi) \widehat{\mathbf{f}}(\xi) $ is a matrix-vector product.
Similarly, $p \odot q (x,\xi)=\sum_{|\alpha|=0}^{\infty}\frac{(-i)^{|\alpha|}}{\alpha !}D_{\xi}^{\alpha}p(x,\xi)D_x^{\alpha}q(x,\xi)$ involves matrix products.
This class is closed under $\odot$ and addition.
Differential operators have symbols which are polynomial in $\xi$, thus their symbols are polyhomogeneous. 

For a pseudodifferential operator $P$ with   symbol $p= \sum_{j=0}^{\infty} p_j \in S^N(X,m)$, the {\em principal symbol} 
is  ${p_0} \in S^N(X,m)$. 
Although this is an equivalence class of symbols,  we make the slight abuse of terminology  by referring to ``the'' principal symbol 
and we denote it by $\sigma(P):=p_0$. 
We note especially that the values of $\sigma(P)$ for small values of $\xi$ are unimportant, and so we generally
give $\sigma(P)(x,\xi)$ only for $|\xi|\ge 1$.

\paragraph{Ellipticity and parametrices}
The property that ensures existence of a parametrix is {\em ellipticity} of the symbol. 
We use the following definition, which is restrictive 
-- a more robust definition would be valid for symbols in $S_{1,0}^N(O)$ -- 
but it is sufficient for our purposes.
\begin{definition}
A  symbol $p\in S^N(U,m)$ is elliptic 
if
$p_0(x,\xi)$ is non-singular for $|\xi|\ge 1$.
\end{definition}
Note in particular that if $p_0$ is as positively homogeneous, scalar symbol of order $N$ which does not vanish, then there
is a constant $c>0 $ so that $c |\xi|^N \le |p_0(x,\xi)|$, and therefore $ |\xi|^{-N} |p(x,\xi)|$ is bounded below for $|\xi|$ sufficiently large.
The following consequence of ellipticity is a simplification (sufficient for our purposes) of \cite[Theorem 7.18]{Grubb}.

\begin{lemma}\label{local_parametrix_lemma}
When $p$ is elliptic 
there is a properly supported pseudodifferential operator $Q$ (with symbol $q\in S^{-N}(O,m)$, modulo $S^{-\infty}(O,m)$) so that 
$Q\Op(p) - \mathrm{Id}$ and $\Op(p) Q - \mathrm{Id}$ are smoothing operators.
\end{lemma}
\begin{proof} 
We begin by constructing a left parametrix $Q$, so that $Q\Op(p) - \mathrm{Id}\sim 0$. 
Constructing a right parametrix follows with an obvious modification.
The identification of right and left parametrices (modulo a smoothing operator) follows by associativity.

Let $\psi:O\times \reals^d\to [0,1]$ be a smooth function in $ C^{\infty}(O\times \reals^d)$ 
which satisfies that $\psi(x,\xi) = 0$  in the singular set 
$\{(x,\xi)\in O\times \reals^d\mid \det({p_0}(x,\xi) ) = 0\}$, 
and 
so that $\psi(x,\xi) =1$ outside of a neighborhood of this set. 
Let 
$\tilde{q}(x,\xi) = {\psi(x,\xi)}\bigl({p_0}(x,\xi)\bigr)^{-1}$, 
and note that 
 $\tilde{q}$ is positive homogeneous by ellipticity of ${p_0}$ and the product rule. 
The product
$\tilde{q}\odot p=1+r$ is in $S^{0}(O,m)$, 
with  $r = \sum_{|\alpha|=1}^{\infty}\frac{(-i)^{|\alpha|}}{\alpha !}D_{\xi}^{\alpha}\tilde{q}(x,\xi)D_x^{\alpha}{p}(x,\xi)$.
 Indeed,  we have $D_{\xi}^{\alpha}\tilde{q}(x,\xi)D_x^{\alpha}{p}(x,\xi)\in S^{-N-|\alpha| +N}(O,m)$ for 
each $|\alpha|\ge 1$,
so  $r  \in S^{-1}(O,m)$.
The symbol $r^{\odot j} := r\odot \dots \odot r$ ($j$-fold product)  resides in $S^{-j}(O,m)$, and the  
Neumann series $\sum_{j=0}^{\infty}  \bigl(-r\bigr)^{\odot j}\in S^0(O,m)$, give  the desired parametrix symbol.
The desired parametrix is 
 $Q \sim \mathrm{Op}\bigl(\sum_{j=0}^{\infty}  (-r)^{\odot j} \odot \tilde{q}\bigr) $.
\end{proof}

It follows from the construction  that if $\sum_{j=0}^N p_j$ 
is the symbol of an elliptic differential operator (with $p_j\in S^{N-j}(O)$), 
then the parametrix $Q$ of $\mathrm{Op}(p)$ has  symbol $q = \sum_{j=0}^{\infty} q_j\in S^{-N}(O)$.
We can say more, however: that each term $q_j$ is rational in $\xi$.

\begin{lemma}\label{phg_parametrix}
Suppose $p$   is the symbol of a (scalar) elliptic differential operator of order $N$.
Then its parametrix $Q = \Op(q)$ is polyhomogeneous, with $q = \sum_{j=0}^{\infty}q_j$.
Moreover,  for every $j$, $q_j$ is positively homogeneous of order $-N-j$, 
and for $|\xi|\ge 1$, $\xi\mapsto q_j(x,\xi)$ is a rational function. 
\end{lemma}
\begin{proof}
We write $p = \sum_{j=0}^N p_j$, so that each $p_j\in S^{N-j}$ 
is a homogeneous polynomial of degree $N-j$, and therefore satisfies (\ref{weak_polyhomogeneity}). 

The  terms of $q$ can be determined via the product formula 
$ \left(\sum_{j=0}^{\infty} q_j\right) \odot  \left(\sum_{j=0}^N p_j \right)= 1$.
Namely, after rearranging terms, we have the equation
$$  \left(\sum_{j=0}^{\infty} q_j\right) \odot  \left(\sum_{j=0}^N p_j \right) 
= 
\sum_{j=0}^{\infty} \sum_{|\alpha|+k+\ell = j} \frac{(-i)^{|\alpha|}}{\alpha!}D_{\xi}^{\alpha}q_k(x,\xi) D_x^{\alpha}p_{\ell}(x,\xi).
$$
With the aid of a cutoff function, set $q_0 (x,\xi) = (p_0(x,\xi))^{-1}$  for $|\xi|\ge 1$, 
and note that for $|\xi|\ge 1$, this is  rational and positively homogeneous of order $-N$.
Each term $ \frac{(-i)^{|\alpha|}}{\alpha!}D_{\xi}^{\alpha}q_k(x,\xi) D_x^{\alpha}p_{\ell}(x,\xi)$  is a symbol of order 
$-(k+\ell+|\alpha|)=-j$.
Proceed by induction on $j$, 
setting (for $|\xi|\ge 1$)
\begin{equation}\label{parametrix_expansion}
q_j (x,\xi) = -(p_0(x,\xi))^{-1} \sum_{k=0}^{j-1}\left( \sum_{\ell+|\alpha| 
= 
j-k}\frac{(-i)^{|\alpha|}}{\alpha!}D_{\xi}^{\alpha}q_k(x,\xi) D_x^{\alpha}p_{\ell}(x,\xi)\right).
\end{equation}
Then $D_{\xi}^{\alpha}q_k(x,\xi)$ is rational and positively homogeneous of order $-N -k -|\alpha|$, while 
$ D_x^{\alpha}p_{\ell}(x,\xi)$ is polynomial and positively homogeneous of order $N-\ell$ in $\xi$. 
Thus $\frac{(-i)^{|\alpha|}}{\alpha!}D_{\xi}^{\alpha}q_k(x,\xi) D_x^{\alpha}p_{\ell}(x,\xi)$ is positive homogeneous of order
$-k - \ell - |\alpha| = -j$ for $k<j$, so  $q_j$ is positively homogeneous of order $-N -j$.
It is likewise rational as a sum of rational functions.
\end{proof}

 \subsection{Expression of operators in coordinates}

In this section, we express the basic operators under consideration in normal and tangential coordinates near the boundary.
Namely, we consider a map $\Psi':O'\to U'$, with $O'\subset \reals^{d-1}$ and $U'\subset \partial \Omega$ as described in Section \ref{notation}.
We calculate the effect of the diffeomorphism 
$\Psi:O\to U$ given in (\ref{explicit_normal_tangential}) on
the Laplacian, the boundary operators $\Lambda_j$, and 
the fundamental solution of $\Delta^m$. 
Finally, we use this to analyze the boundary layer potential operators $V_j$.  

\paragraph{Laplace operator} 
From the decomposition (\ref{Loc_Lap}), the principal symbol for ${\Delta^{\Psi}}$ is 
$$
\sigma({\Delta}^{\Psi})(\bfx,\xi) =
\sigma({\Delta_{t}}^{\Psi} + 
\frac{\partial^2}{\partial x_d^2})(\bfx,\xi) 
=
-\left(\xi_d^2+\sum_{j=1}^{d-1} \sum_{k=1}^{d-1} 
 \sfg^{i,j}(\bfx)\xi_j \xi_k\right),$$
and because of the positivity of the first (i.e., least) eigenvalue of $\sfg^{-1}$ 
we see that ${\Delta}^{\Psi}$ and $({\Delta}^m)^{\Psi}$ 
are elliptic (of order 2 and $2m$ respectively). 
Writing $\bfeta  = (\xi_1,\dots,\xi_{d-1})$, we denote
 $$
 \sfd(\bfx, \bfeta) :=\sum_{j=1}^{d-1} \sum_{k=1}^{d-1} 
 \sfg^{j,k}(\bfx)\xi_j \xi_k,
 $$
 which allows us to write
 $$\sigma\bigl(({\Delta})^{\Psi}\bigr)(\bfx,\xi) = -\left(\xi_d^2+  \sfd(\bfx, \bfeta)  \right).$$
 At times, we will consider $\sfd|_{O'\times \reals^d}$, 
 and we express this restriction as $ \sfd(\bfy, \cdot)$ which simply means $\sfd(\bfx,\cdot)$ with $x_d=0$.

\paragraph{Normal derivative} 
The principal symbol of  $({ D_{\vec{n}}^{\t}})^{\Psi}$ is
$\sigma\bigl(({ D_{\vec{n}}^{\t}})^{\Psi}\bigr)(\bfx,\xi) = - i \xi_d$.
We can also express the principal symbol of the differential operator $\Lambda_j^{\t}$ -- 
the adjoint of the operator defined
in (\ref{extended_operator_def}) -- as
\begin{equation}\label{boundary_diff_op_ps}
\sigma\bigl((\Lambda_j^{\t} )^{\Psi})\bigr)(\bfx,\xi)=
\begin{cases}
(-1)^{\frac{j}{2}}
  \left(\xi_d^2+
   \sfd(\bfx, \bfeta) 
\right)^{ j/2}& j \ \text{is even}\\
(-1)^{\frac{j+1}{2}} i\xi_d
 \left(\xi_d^2
 + \sfd(\bfx, \bfeta) 
\right)^{ (j-1)/2}& j \ \text{is odd}.
\end{cases}
\end{equation}

\subsubsection{The fundamental solution to $\Delta^m$ in local coordinates} 
The solution operator, $f\mapsto \phi*f$, for $\Delta^m$ in $\reals^d$ 
 is a Fourier multiplier with symbol $\widehat{\phi}(\xi) = |\xi|^{-2m}$ 
(at least, when considering distributions supported on $\reals^d\setminus \{0\}$).
If not for its behavior near $\xi=0$
it would be in
$S^{-2m}(\reals^d)$. 
This is easily fixed by making the decomposition
$\phi*f = Ef +Kf$, 
into a properly supported pseudodifferential operator and a smoothing operator.

Note that  the formula $ \Delta^m \phi* g= g =  \phi*(\Delta^m g)$ is 
valid for test functions $g\in\D(\reals^d)$ which satisfy $g\perp \Pi_{2m}$. 
Thus 
$({\Delta}^m)^{\Psi} (E)^{\Psi}$ and $ (E)^{\Psi} ({\Delta}^m)^{\Psi} $  
both equal the identity, modulo addition of a smoothing operator.
It follows that 
$E^{\Psi}$ is 
a parametrix for  
$({\Delta}^m)^{\Psi}$, the $m$-fold composition of
the operator $({\Delta})^{\Psi}$ from (\ref{Loc_Lap}) on $O$ (derived from $\Delta^m$).
By Lemma \ref{phg_parametrix}, $E^{\Psi}$ has a polyhomogeneous symbol $\sum_{j=0}^{\infty}e_j(\bfx,\bfxi)$;
 we can express its principal symbol   as
\begin{equation}\label{fs_ps}
\sigma(E^{\Psi})(\bfx,\bfxi) = e_0(\bfx,\bfxi) = (-1)^m \left(\xi_d^2+
 \sfd(\bfx, \bfeta) 
 \right)^{-m}.
 \end{equation}
 \begin{remark}
 We note that the symbol for $E^{\Psi}$ could be obtained by the change of variables formula in item 4 of Result \ref{results}.
In particular,   the principal symbol could have been determined the transformation for $p\in S^m$
to 
$
p(\Psi (\bfx), {\bf M}(\bfx)\bfxi),
$  
where ${\bf M}(\bfx) = (\bigl[D\Psi(\bfx)\bigr]^{-1})^{\t}$. 
The above result gives 
$e_0(\bfx,\bfxi) =- |{\bf M}(\bfx)\bfxi|^2 =
-\bfxi^{\t} \sfg^{-1}(\bfx)\bfxi$,
since  the symbol
for the Laplacian in standard coordinates is $p(u,\zeta)  = -|\zeta|^2$
and
the  Gram matrix is $[D\Psi(\bfx)]^{\t}D\Psi(\bfx)$ 
and, so, the inverse Gram matrix is 
$\sfg^{-1}(\bfx)= {\bf M}(\bfx)^{\t}{\bf M}(\bfx)$. 
\end{remark}
\paragraph{\bf The boundary layer potential operator in local coordinates}
To describe the coordinate representation of the operators 
$g\mapsto \Lambda_k V_j g= \Lambda_k \phi* \bigl(\Lambda_j^{\t} g \cdot \delta_{\partial \Omega}\bigr)$, 
we  focus on the  the coordinate version  of  
$(\Lambda_k) E (\Lambda_j^{\t})$,
since this differs from the map 
$f\mapsto \Lambda_k  \phi* ( \Lambda_j^{\t} f) $  
 by a smoothing operator. 
It follows that, it too is polyhomogeneous (as  a product of polyhomogeneous operators):
$(\Lambda_k) E (\Lambda_j^{\t}) = \mathrm{Op}(p)$, with $p= \sum_{\ell=0}^{\infty}p_\ell$, and $p_\ell\in S^{j+k-2m-\ell}(O)$.
Writing $n = j+k$, its  principal symbol is determined by combining (\ref{boundary_diff_op_ps}) and (\ref{fs_ps}):
$$
\sigma\left(
\bigl(\Lambda_k
E
 \Lambda_j^{\t}\bigr)^{\Psi}
\right)(\bfx,\bfxi) 
 =
(-1)^{m-\frac{n}{2}} \begin{cases}
\frac{1}{\left(\xi_d^2+
 \sfd(\bfx, \bfeta)
 \right)^{m-\frac{n}{2}}}&\quad j, k \text{ are even, }\\
 \mbox{}&\mbox{}\\
 \frac{\xi_d^2}{\left(\xi_d^2+
  \sfd(\bfx, \bfeta)
 \right)^{m+1-\frac{n}{2}}}&\quad j, k \text{ are odd, }\\
 \mbox{}&\mbox{}\\
\frac{i(-1)^{\ell}\xi_d}{\left(\xi_d^2+
  \sfd(\bfx, \bfeta)
 \right)^{ m-\frac{n-1}{2}}}&\quad n \text{ is odd.} 
 \end{cases}
$$

%
%
%
%
%
%
\subsection{Boundary operators in local coordinates}\label{BOLC}
The expression of $(\Lambda_k E \Lambda_j^{\t})^{\Psi}$ as the polyhomogeneous operator $p$
 permits us to write 
$(v_{k,j}^{\pm})^{\Psi}$
as an operator from $\D(O')$ to $\E(O')$.
To this end, define the operator $\widetilde{v_{k,j}^{\pm}} $ as 
$$\widetilde{v_{k,j}^{\pm}} g (\bfy)
  :=
  \lim_{x_d\to 0^{\pm}}  (\Lambda_k E  \Lambda_j^{\t})^{\Psi} (g\cdot \delta_{\reals^{d-1}}) (\bfx).$$
  That this is well defined for 
  smooth $g$
  is an immediate consequence of Corollary \ref{boundary_smoothness};
  indeed, we can 
 write $(v_{k,j}^{\pm})^{\Psi} $ as the sum of the
  operator $ \widetilde{v^{\pm}}_{k,j}$ and a smoothing operator.
  Namely
 \begin{equation}\label{pdbo_in_coords}
 (v_{k,j}^{\pm})^{\Psi}g(\bfy)
  =
   \widetilde{v_{k,j}^{\pm}} g (\bfy)
+\lim_{x_d\to 0}\left[  
 { \Lambda_{k}}^{\Psi} K  {\Lambda_j^{\t}}^{\Psi} \bigl(g \cdot\delta_{\reals^{d-1}}\bigr)
 \right](\bfx).
 \end{equation}
Note that when $j+k\le 2m-2$, 
$ (v_{k,j}^{+})^{\Psi}g(\bfy)=  (v_{k,j}^{-})^{\Psi}g(\bfy)$,
so we simply write  $\widetilde{v_{k,j}^{\pm}} g (\bfy) =\widetilde{v_{k,j}} g (\bfy)$.
 
 The following lemma shows that $\widetilde{v_{k,j}^{\pm}}$, and hence $v_{k,j}^{\pm}$ can be extended to distributions.
It shows, roughly, that $(\Lambda_k E  \Lambda_j^{\t})^{\Psi}$ has the ``transmission property'',
 and will be clear to those familiar with this subject. 
 In this case, it is fairly easy to demonstrate, since we are considering
  symbols with convenient analytic extension. 
 The structure of this proof (especially Case 2) follows Section 18.2 of \cite{Horm3}, 
 although by dealing with classical symbols, it is greatly simplified.
\begin{lemma}
For any $j,k\in \nats$, $ \widetilde{v_{k,j}^{\pm}}$ is a polyhomogeneous pseudodifferential operator of order
$j+k-2m+1$.%
\end{lemma}
\begin{proof}
We split this into two cases.
   
{\bf Case 1: $j+k \le 2m -2$.}
In this case we have  
$$ 
(\Lambda_k E  \Lambda_j^{\t})^{\Psi} (g\cdot \delta_{\reals^{d-1}}) (\bfx)
     = (2\pi)^{-d}
 \int_{\reals^{d-1}} 
  \widehat{g}(\bfeta) 
     e^{i\bfy\cdot \bfeta}
\left( \int_{\reals}
  p(\bfx, \bfeta, \xi_d) 
    e^{ix_d\cdot \xi_d} \dif \xi_d 
    \right)\dif \eta    
     $$
(the inner integral is convergent by decay of $p$).

Allowing $x_d\to 0$,  we have
$\widetilde{v_{k,j}} g(\bfy) =  
(2\pi)^{-d} \int_{\reals^{d-1}} 
  \widehat{g}(\bfeta) 
     e^{i\bfy\cdot \bfeta}
\left( \int_{\reals}
  p(\bfy,0, \bfeta, \xi_d) 
 \dif \xi_d 
    \right)\dif \eta  $
    after exchanging limit and integral. 
The conditions on  
$p$  
ensure that
 $(\bfy,\bfeta) \mapsto 
 \frac{1}{2\pi} \int_{\reals}
  p(\bfy,0, \bfeta, \xi_d) 
 \dif \xi_d$
 is a symbol in 
 $S^{j+k-2m+1}(O')$. Indeed, for multi-indices $\alpha$ and $\beta$, we 
 have (by dominated convergence)
 \begin{eqnarray*}
 |D_{\bfy}^{\alpha}D_{\bfeta}^{\beta}\int_{\reals}
  p(\bfy,0, \bfeta, \xi_d) 
 \dif \xi_d| 
 &\le&
  \int_{\reals}
|( D_x^{\alpha}D_{\xi}^{\beta} p)(\bfy,0, \bfeta, \xi_d)|
 \dif \xi_d\\
 &\le & C   \int_{\reals}
(1+|\bfeta|+|\xi_d|)^{j+k-2m -|\beta|}
 \dif \xi_d \\
 &=& C (1+|\bfeta|)^{j+k-2m+1-|\beta|}
\end{eqnarray*}
where in the last equation, we have used the change of variable $t = \frac{\xi_d}{1+|\bfeta|}$.

A similar estimate applied to each $p_{\ell}$ guarantees that we can express $(v_{k,j})^{\Psi}$ as a  polyhomogeneous series, namely,
$$  \int_{\reals}
  p(\bfy,0, \bfeta, \xi_d) 
 \dif \xi_d 
 =
 \sum_{\ell=0}^{\infty}
\int_{\reals}
  p_{\ell}(\bfy,0, \bfeta, \xi_d) 
 \dif \xi_d.
      $$
 For $|\eta|>1$ and $\lambda>1$,
 $\int_{\reals}
  p_{\ell}(\bfy,0, \lambda \bfeta, \xi_d) 
 \dif \xi_d 
 = 
 \lambda^{ j+k-2m-\ell +1}
 \int_{\reals}
  p_{\ell}(\bfy,0,  \bfeta, \zeta) 
 \dif \zeta $ by a simple change of variable, so
 each term is positively homogeneous of order $j+k-2m-\ell +1$.

{\bf Case 2: $j+k> 2m-2$.}
In this case, we  write $p = \sum_{\ell=0}^Np_j + {p}^{\flat}$, choosing $j+k-2m -N-1\le 2$. 
This permits
us to treat ${p}^{\flat}$ as in Case 1; this is left to the reader.
We focus on $p^{\sharp} =  \sum_{\ell=0}^Np_j + {p}^{\flat}$.

Consider the case $x_d>0$, the other case is handled similarly.
Begin by mollifying $g\cdot \delta$ as follows:
 for a smooth $\tau:\reals\to \reals$ supported in $[-1,1]$, 
 consider $G_{\epsilon} (\bfx)= \frac{1}{\epsilon}g(\bfy) \tau(x_d/\epsilon)$. 
 Then
 $(\Lambda_k E  \Lambda_j^{\t})^{\Psi} (g\cdot \delta_{\reals^{d-1}}) 
 = 
 \lim_{\epsilon\to 0} (\Lambda_k E  \Lambda_j^{\t})^{\Psi} G_{\epsilon}$.

Because $\widehat{G_{\epsilon}}(\bfxi)=\widehat{g}(\bfeta) \widehat{\tau}(\epsilon \xi_d)$ is a Schwartz function,
 it follows that 
 \begin{equation}\label{transmission_fubini}
 \Op(p^{\sharp})(g\cdot \delta_{\reals^{d-1}})(\bfx)
 =
 (2\pi)^{-d}
 \int_{\reals^{d-1}}\widehat{g}(\eta)e^{i\langle\bfy,\bfeta\rangle}\int_{\reals} 
 \widehat{\tau}(\epsilon \xi_d)
 p^{\sharp}  
 (\bfx,\bfeta,\xi_d)
 e^{ix_d \xi_d}
 \dif \xi_d \dif \bfeta
\end{equation}
because $\bfxi\mapsto \widehat{G_{\epsilon}}(\bfxi) p^{\sharp}  (\bfx,\bfxi)$  
is integrable.

Note that $\widehat{\tau}$ is defined on $\comps$ and is entire. 
Because each $p_{\ell}$ is rational in $\xi$ (for $|\xi|>1$),
there is a complex region
$\Omega_{R_0} := \{\zeta\in \comps\mid |\zeta|>R_0, \Im(\zeta)>0\}$ where  
for each  $\ell = 0\dots N$,
$\zeta \mapsto p_{\ell}(\bfx,\bfeta,\zeta)$ 
is defined and analytic. 
The inner integral 
$\int_{\reals} \widehat{\tau}(\epsilon \xi_d)p^{\sharp} 
(\bfx,\bfeta,\xi_d)e^{ix_d \xi_d}\dif \xi_d 
$  
in (\ref{transmission_fubini}) can be written as
$$
\int_{-R}^R \widehat{\tau}(\epsilon \xi_d) 
p^{\sharp}
(\bfx,\bfeta,\xi_d)e^{ix_d \xi_d}\dif \xi_d -
 \int_{\gamma_R} \widehat{\tau}(\epsilon \zeta)
 p^{\sharp} 
 (\bfx,\bfeta,\zeta)e^{ix_d \zeta}\dif \zeta
$$
for any $R_0<R<\infty$.
(Here $\gamma_R$ is the upper  part of semicircle of radius $R$ centered at $0$.) Not
Because $|e^{ix_d \zeta} \widehat{\tau}(\epsilon \zeta)| = | \int_{-1}^1 \tau(t) e^{i(x_d-\epsilon t)\zeta} \dif t|$, 
we have that 
$|e^{ix_d \zeta} \widehat{\tau}(\epsilon \zeta)| \le \|\tau\|_1$ 
provided $\epsilon<x_d$ and $\Im \zeta \ge 0$. 
By dominated convergence, we 
then have that 
 \begin{multline*}  
 \Op(p^{\sharp})(g\cdot \delta_{\reals^{d-1}})
  =\\
  (2\pi)^{-d}\int_{\reals^{d-1}}\widehat{g}(\eta)e^{i\langle\bfy,\bfeta\rangle}\left(\int_{-R}^R  p^{\sharp} 
  (\bfx,\bfeta,\xi_d)e^{ix_d \xi_d}\dif \xi_d -
 \int_{\gamma_R} p^{\sharp} 
 (\bfx,\bfeta,\zeta)e^{ix_d \zeta}\dif \zeta\right) \dif \bfeta
\end{multline*}
 Applying dominated convergence again as we let $x_d\to 0^+$, we have
\begin{multline*} 
\widetilde{v_{k,j}^{\pm}} g (\bfy) = \\
    (2\pi)^{-d} \int_{\reals^{d-1}}\widehat{g}(\eta)e^{i\langle\bfy,\bfeta\rangle}
   \left(
     \int_{-R}^R p^{\sharp} 
       (\bfy,0,\bfeta,\xi_d)
     \dif \xi_d 
     -
     \int_{\gamma_R} p^{\sharp} 
       (\bfy,0,\bfeta,\zeta)
      \dif \zeta\right) 
     \dif \bfeta.
 \end{multline*}
 The fact that the symbol is a  positively homogeneous symbol follows by
 considering for $\lambda \ge 1$ and each $\ell = 0 \dots N$, the integral
 $$
 \int_{-\lambda R}^{\lambda R} p_{\ell}(\bfy,0,\lambda \bfeta,\xi_d)\dif \xi_d -
 \int_{\gamma_{\lambda R}}p_{\ell}(\bfy,0,\lambda \bfeta,\zeta)\dif \zeta,
$$ and applying a change of variable as  in Case 1.
 \end{proof}

When $j+k=n\le 2m-2$, the principal symbol, in local coordinates, is (for $|\bfeta|\ge 1$)
the convergent integral
\begin{equation*}
\sigma\left(   \widetilde{v_{k,j}}\right)(\bfy,\bfeta) =\frac{(-1)^{m-\frac{n}2} }{2\pi}
\begin{cases}
\int_{-\infty}^{\infty}\frac{1}{\left(\xi_d^2+
 \sfd(\bfy, \bfeta)
\right)^{m- \frac{n}{2}}}\dif \xi_d& j, k \text{  both even,}\\
\mbox{}&\mbox{}\\
 \int_{-\infty}^{\infty}
 \frac{\xi_d^2}{\left(\xi_d^2+
 \sfd(\bfy, \bfeta)
\right)^{m+1 - \frac{n}{2}}}
\dif \xi_d
& j, k \text{  both odd,}\\
\mbox{}&\mbox{}\\
\int_{-\infty}^{\infty}\frac{-i\xi_d}
{\left(\xi_d^2+\sfd(\bfy, \bfeta)\right)^{ m- \frac{n-1}{2}}}\dif \xi_d & n = j +k\text{ odd.} 
 \end{cases}
\end{equation*} 
After a change of variable, and integrating out the $\xi_d$ variable,
we are left with the simple expression 
\begin{equation}\label{princ_symb}
\sigma\left(    \widetilde{v}_{k,j}\right)(\bfy,\bfeta) = (-1)^{m-n} C_{k,j} \sfd(\bfy, \bfeta)^{\frac{n+1}{2} -m}.
\end{equation}
with constant
\begin{equation}\label{coeffs}
C_{k,j}  = 
\frac{1}{2\pi}
\begin{cases}
\int_{-\infty}^{\infty}\frac{1}{\left(\zeta^2+
1\right)^{m- \frac{n}{2}}}\dif \zeta =  2^{1+n-2m} \mathfrak{b}_{m-n/2-1} 
& j, k \text{ both even,}\\
\mbox{}&\mbox{}\\
 \int_{-\infty}^{\infty}\frac{\zeta^2}{\left(\zeta^2+1 \right)^{m+1 - \frac{n}{2}}}
\dif \zeta = 
 2^{1+n-2m}
\mathfrak{c}_{m-n/2-1}
& j, k \text{ both odd,}\\
\mbox{}&\mbox{}\\
\int_{-\infty}^{\infty}\frac{-i\zeta}
{\left(\zeta^2+1\right)^{ m- \frac{n-1}{2}}}\dif \zeta=0 & n = j +k
\text{ odd.} 
 \end{cases}
 \end{equation}
 where $\mathfrak{b}_j := \frac{(2j)!}{j!j!}$ 
 are middle binomial coefficients and 
 $\mathfrak{c}_{j-1} := 4\mathfrak{b}_{j-1} - \mathfrak{b}_j$ are double Catalan numbers
 (this follows from \cite[Proposition 4.1]{HL}).

\begin{proof}[Proof of Lemma \ref{smoothing_lemma}]
Let $(U_{\ell},\Phi_{\ell})_{\ell=1\dots N}$ be an atlas for $\partial \Omega$, with  $\Phi_{\ell}:U_{\ell} \to O_{\ell}\subset \reals^{d-1}$  a diffeomorphism
and $\Psi_{\ell} = \Phi_{\ell}^{-1}$.
Let $(\tau_\ell)_{\ell=1\dots N}$ be a smooth partition of unity for $\partial \Omega$ subordinate to $(U_\ell)_{\ell=1\dots N}$,
and let $(\tilde{\tau}_{\ell})_{\ell=1\dots N}$ be a family of  smooth cut-off functions  so that
$\tilde{\tau}:\partial \Omega \to [0,1]$ with $\supp{\tilde{\tau}_{\ell}}\subset U_\ell$ and $\tilde{\tau}_{\ell}(z) =1$ for $z\in \supp{\tau_{\ell}}$.

 Let $\sigma = s+2m-j-k-1$.
 We wish to show that for all $\ell=1\dots N$, 
 $$\| \Psi_{\ell}^*( \tau_{\ell} v_{k,j}f)\|_{H_p^{\sigma}(O_\ell)} 
 \le 
 C \sum_{\ell'=1}^N \|\Psi_{\ell'}^*(\tau_{\ell'} f)\|_{H_p^{\sigma +j+k+1-2m}(O_{\ell'})}
 $$
holds.

For $f\in H_p^s(\partial \Omega)$ we 
have $v_{k,j} f = \sum_{\ell'=1}^Nv_{k,j} \tau_{\ell'} f$,
and we consider $\tau_{\ell}v_{k,j}\tau_{\ell'} f$ for each value of 
$\ell$ and $\ell'$. We further split this using the functions $\tilde{\tau}_{\ell'}$, obtaining 
the maps $b_{\ell,\ell'} f:= \tilde{\tau}_{\ell'}\tau_{\ell}v_{j,k}\tau_{\ell'}f$ and
$g_{\ell,\ell'} f:=(1-\tilde{\tau}_{\ell'})\tau_{\ell}v_{k,j}\tau_{\ell'}f$. Clearly, 
$$\tau_{\ell}v_{k,j} f =  \sum_{\ell'=1}^N b_{\ell,\ell'}f + \sum_{\ell'=1}^N g_{\ell,\ell'}f.$$

Let $\mathcal{W}_{\ell,\ell'}: = \mathrm{supp}\bigl((1-\tilde{\tau}_{\ell})\times\tau_{\ell'}\bigr)$ and $\mathcal{U}_{\ell} := \supp{\tau_\ell}$.
Because $\mathcal{W}_{\ell,\ell'}$ and $\mathcal{U}_{\ell}$  
are disjoint, we have that the map
$\mathcal{D}'( \mathcal{U}_{\ell} ) \to \mathcal{D}'(\mathcal{W}_{\ell,\ell'}):F\mapsto \left(v_{k,j}F\right)|_{\mathcal{W}_{\ell,\ell'}}$ 
has a $C^{\infty}$ kernel. Namely, the restriction of  $(x,\alpha) \mapsto \lambda_{k,x} \lambda_{j,\alpha} \phi(x-\alpha)$
to $\mathcal{W}_{\ell,\ell'} \times \mathcal{U}_{\ell}$. 
Thus $\|\Psi_{\ell}^*g_{\ell,\ell'} f\|_{H_p^{\sigma_1}(\partial\Omega)} \le C\|\Psi_{\ell}^* (\tau_{\ell}f)\|_{H_p^{\sigma_2}(\partial\Omega)}$ 
for any pair $\sigma_1,\sigma_2\in \reals$.

The boundedness of $b_{\ell,\ell'}$ follows from the estimate
\begin{eqnarray*}
\| \Psi_{\ell}^*(\tilde{\tau}_{\ell'} \tau_{\ell}v_{k,j}  \tau_{\ell'}f)\|_{H_p^{\sigma}( O_{\ell})}  
&\le& C\| \Psi_{\ell'}^*(\tilde{\tau}_{\ell'} \tau_{\ell}v_{k,j} \tau_{\ell'}f) \|_{H_p^{\sigma}( O_{\ell'})}\\
&\le& C\|\widetilde{v_{k,j}^{\pm}} \Psi_{\ell'}^*( \tau_{\ell'}f) \|_{H_p^{\sigma}(O_{\ell'})}\\
&& + C\|\Psi_{\ell'}^*( \tau_{\ell'}f)\|_{H_p^{{\sigma}+j+k+1-2m}(O_{\ell'})}\\
&\le& C \|\Psi_{\ell'}^*( \tau_{\ell'}f)\|_{H_p^{{\sigma}+j+k+1-2m}(O_{\ell'})}.
\end{eqnarray*}
The first inequality follows because $\Psi_{\ell'}^{-1}\circ\Psi_{\ell}$  restricted to 
$\Psi_{\ell}^{-1}(U_\ell\cap U_{\ell'})$ is a diffeomorphism
from $\Psi_{\ell}^{-1}(U_\ell\cap U_{\ell'})$ to $\Psi_{\ell'}^{-1}(U_\ell\cap U_{\ell'})$.
The second follows from (\ref{pdbo_in_coords}) and the fact that 
$  { \Lambda_{k}}^{\Psi} K  {\Lambda_j^{\t}}^{\Psi}$  
is smoothing.
The last inequality follows because $\widetilde{v_{k,j}^{\pm}}$ is a pseudodifferential operator 
with symbol in $S^{j+k+1-2m}(O_{\ell'})$
and therefore is bounded from $H_{p}^{\sigma}(O_{\ell'})$ to $H_p^{\sigma+2m-j-k-1}(O_{\ell'})$
(this follows from the mapping properties of pseudodifferential operators \cite[Ch. 11, Theorem 2.1]{Tayl}).
\end{proof}

%
%
\subsection{Ellipticity of Matrix Symbols and a Parametrix}
\label{parametrix}
In this section we construct the global right parametrix $R$ for $L$. 
This is done in two stages -- first by generating a local parametrix in coordinates on $O'\subset \reals^{d-1}$
by way of Lemma \ref{local_parametrix_lemma}, and then by carefully piecing together a number of local parametrices with the aid of a partition of unity.

\subsubsection{A local parametrix}\label{local_parametrix}
We consider again a map $\Psi':O'\to U'$, with $O'\subset \reals^{d-1}$ and $U'\subset \partial \Omega$. 
Let $M_1 := \min_{\bfy\in O'} \lambda_1(\bfy)$ and $M_{d-1} := \max_{\bfy\in O'} \lambda_{d-1}(\bfy)$ the least 
 and greatest eigenvalues, respectively, of the  inverse  Gram matrix $\sfg^{-1}(\bfy)$ described in Section \ref{notation}.

The operator
\begin{equation}\label{local_matrix_operator}
\widetilde{\LL} : \begin{pmatrix}s_0\\ \vdots \\ s_{m-1}\end{pmatrix} \mapsto 
\begin{pmatrix}
\widetilde{v}_{00} &\dots& \widetilde{v}_{0, m-1}\\ 
\vdots& \ddots&\vdots\\ 
\widetilde{v}_{m-1,0}& \dots & \widetilde{v}_{m-1,m-1}
\end{pmatrix}
\begin{pmatrix}s_0\\ \vdots \\ s_{m-1}\end{pmatrix},
\end{equation}
is, modulo a smoothing operator,
 the local version of the operator $\LL$.
Its  $k,j$ entry (with
 indices running from $j,k=0\dots m-1$) is a pseudodifferential operator of order 
 $1+j+k -2m$.
 Such operators, having orders that are $m\times m$ Hankel
 matrices are so-called Douglis--Nirenberg elliptic systems, cf. \cite{DN}.

We consider matrix pseudodifferential operators $A$ with symbols having entries $a_{k,j} \in S^0$, 
since such operators have a simple notion of ellipticity. 
The principal symbol $\sigma(A)$ is nonsingular
for large $|\xi|$ if and only if the scalar symbol $\det(\sigma(A)$
is elliptic of order 0.
 Thus, it suffices to check that the determinant of the principal symbol is bounded from below as $|\xi|\to \infty$.

Operators with diagonal symbols comprise another class of operators with a simple notion of ellipticity.
In this case, writing $a(x,\xi) = \bigl(a_{kj}(x,\xi)\bigr)$, 
the off-diagonal entry $a_{kj}(x,\xi)$ (with $j\ne k$) is zero and each
diagonal entry $a_{jj}$ is elliptic.  Because such systems are decoupled, a parametrix of the same type exists -- namely
$b(x,\xi) = \bigl(b_{jk}(x,\xi)\bigr)$ with $b_{jj}$ the (scalar) parametrix of $a_{jj}$ and
$b_{jk} = 0$ when $j\ne k$.

Returning to the operator $\widetilde{L}$, we make the decomposition
$\widetilde{L }= \mathsf{A} \mathsf{L}  \mathsf{S}$, with properly supported pseudodifferential operators 
$\mathsf{A}$ and $\mathsf{S}$ that have diagonal symbols with elliptic entries 
and pseudodifferential operator $\mathsf{L}$ having a matrix symbol 
with entries in $S^0$ and which (as we soon shall see) is elliptic. Specifically,
we require
$$\sigma(\mathsf{A})(\bfy, \bfeta) = 
\begin{pmatrix}
\sfd(\bfy, \bfeta)^{(1-m)/2} &
\dots &0\\
 \vdots&\ddots&\vdots\\
0&\dots
& \sfd(\bfy, \bfeta)^{0}
\end{pmatrix},$$
and
$$
\sigma(\mathsf{S}) (\bfy, \bfeta)= 
\begin{pmatrix}
\sfd(\bfy, \bfeta)^{-m/2}  &\dots &0\\
\vdots &\ddots&\vdots\\
0&\dots& \sfd(\bfy, \bfeta)^{-1/2} 
\end{pmatrix}.
$$
Since $M_1|\bfeta|^2 \le \sfd(\bfy, \bfeta) \le  M_d |\bfeta|^2$, we see that each  diagonal entry is 
elliptic. That is, for $j=0\dots m-1$,  
the diagonal entry $\sigma(\mathsf{A})_{jj}$ is elliptic of order  $1+j-m$ and
 $\sigma(\mathsf{S})_{jj}$ is elliptic of order  $j-m$.
The parametrices for operators $\mathsf{A}$ and $\mathsf{S}$ are the decoupled operators $\mathsf{B}$ and $\mathsf{T}$, respectively. These have symbols
$$\sigma(\mathsf{B}) := 
\begin{pmatrix}
b_{0,0}(\bfy, \bfeta) &\dots &0\\
\vdots &\ddots&\vdots\\
0&\dots& b_{m-1,m-1}(\bfy, \bfeta)
\end{pmatrix}$$
and 
$$
\sigma(\mathsf{T}) := 
\begin{pmatrix}
t_{0,0}(\bfy, \bfeta) &\dots &0\\
\vdots &\ddots&\vdots\\
0&\dots& t_{m-1,m-1}(\bfy, \bfeta)
\end{pmatrix}.
$$
with $b_{jj}(y,\eta)$ the parametrix of $\sfd(\bfy, \bfeta)^{(1+j-m)/2}$ (for $j=0\dots m-1$)
and
similarly with  $t_{jj}(y,\eta) $ the parametrix of 
$\sfd(\bfy, \bfeta)^{(j-m)/2}$ (for $j=0\dots m-1$).

The operator $\mathsf{L}$ is simply defined to be the composition $   \mathsf{B} \widetilde{\LL} \mathsf{T}$, and its principal symbol
can be computed by taking the product $\sigma(\mathsf{B})  \sigma(\widetilde{\LL}) \sigma(\mathsf{T})$. 
Indeed, from
(\ref{princ_symb}) and (\ref{coeffs}) we have
$\sigma(\mathsf{L})_{j,k} \in S^0(O)$ when $j+k$ is even (otherwise it is in $S^{-1}$), and 
for $|\bfeta|\ge1$, we have
$$
\sigma(\mathsf{L})_{j,k}(\bfy, \bfeta) =2^{1+ j+k -2m } \begin{cases}    \mathfrak{b}_{m-(j+k)/2-1} &j,k \text{ both even,}\\
   \mathfrak{c}_{m-(j+k)/2-1}  & j,k \text{ both odd.}
 \end{cases}
$$
Since $\det \bigl(\sigma(\mathsf{L})\bigr) = 2^{m^2}$ by \cite[Proposition 5.2]{HL}
(note that the matrix $\sigma(\mathsf{L})_{j,k}$ differs from the matrix $\mathsf{M}_{k,j}$ of \cite{HL} by 
a change of sign in the odd columns), 
it follows that $\mathsf{L}$ is elliptic and has a parametrix $\mathsf{R}$.
A parametrix for $\widetilde{\LL}$ is then $\widetilde{R}= \mathsf{T}\mathsf{R}\mathsf{B}$, a 
pseudodifferential operator whose $j,k$ entry has order $2m - j-k-1$.

\subsubsection{A global parametrix}
We follow \cite[Theorem 8.6]{Grubb} in combining local parametrices of the various $\widetilde{\LL}$ to obtain a global parametrix $R$ for $L$.

Let $(U_{\ell},\Phi_{\ell})_{\ell=1\dots N}$ be an atlas for $\partial \Omega$, and write $\Psi_{\ell} = \Phi_{\ell}^{-1}:O_{\ell}\to U_{\ell}$.
Let $(\tau_\ell)_{\ell=1\dots N}$ be a smooth partition of unity for $\partial \Omega$ subordinate to $(U_\ell)_{\ell=1\dots N}$.
Consider two families of smooth cut-off functions $(\zeta_{\ell})_{\ell=1\dots N}$ and $(\theta_{\ell})_{\ell=1\dots N}$ so that
$\zeta_\ell:\partial \Omega \to [0,1]$ with $\zeta_\ell(z) =1$ for $z\in \supp{\tau_{\ell}}$ and $\supp{\zeta_{\ell}}\subset U_\ell$
and $\theta_{\ell}:\partial \Omega \to [0,1]$ with $\theta_\ell(z) =1$ for $z\in \supp{\zeta_{\ell}}$ and $\supp{\theta_{\ell}}\subset U_j$.

In each $O_\ell$, let $\widetilde{L}_{\ell}$ denote the operator given by (\ref{local_matrix_operator}).
The construction in Section \ref{local_parametrix} guarantees a right parametrix $\widetilde{R}_{\ell}$; 
for distributions supported in $U_{\ell}$, the change
of coordinates $(\widetilde{R}_{\ell})^{\Phi_{\ell}} = \Phi_{\ell}^*\widetilde{R}_{\ell}\Psi_{\ell}^* $ is well-defined. 
Define the global right parametrix $R$ as
$$
Rf(u) := \sum_{\ell=1}^N   \zeta_{\ell}(u)\bigl[( \widetilde{R}_{\ell} )^{\Phi_{\ell}}(\tau_\ell f)\bigr](u).
$$
We  have 
$LRf 
=
\sum_{\ell=1}^N   L\left(\zeta_{\ell}\bigl[( \widetilde{R}_{\ell} )^{\Phi_{\ell}}(\tau_\ell f)\bigr]\right)
\sim 
\sum_{\ell=1}^N \theta_{\ell} L \left( \zeta_{\ell}\bigl[( \widetilde{R}_{\ell} )^{\Phi_{\ell}}(\tau_\ell f)\bigr]\right)
$.
Note that $\widetilde{L}_{\ell}$ differs from $L^{\Psi_{\ell}}$ on $O_{\ell}$ by a smoothing operator.
A similar statement can be made on $\partial \Omega$: for each $\ell$,  
$\theta_{\ell}L\zeta_{\ell}\sim \theta_{\ell} (\widetilde{L}_{\ell})^{\Phi_{\ell}}\zeta_{\ell} $,
so 
\begin{eqnarray*}
LRf 
&\sim&
\sum_{\ell=1}^N \theta_{\ell}  (\widetilde{L}_{\ell})^{\Phi_{\ell}} \left( \zeta_{\ell}\bigl[( \widetilde{R}_{\ell} )^{\Phi_{\ell}}(\tau_\ell f)\bigr]\right)\\
&\sim&
\sum_{\ell=1}^N \theta_{\ell}  (\widetilde{L}_{\ell})^{\Phi_{\ell}} \left( \bigl[( \widetilde{R}_{\ell} )^{\Phi_{\ell}}(\tau_\ell f)\bigr]\right)\\
&\sim &
\sum_{\ell=1}^N \theta_{\ell} \tau_{\ell} f=f
\end{eqnarray*}
In the second equivalence (modulo a smoothing operator), we have made use of item 1 of Result \ref{results} 
and the fact that $\supp{\tau_{\ell}}$ is contained in the zero set of ${1-\zeta_{\ell}}$.
\begin{proof}[Proof of Lemma \ref{closed}]
Because there exists a right parametrix for 
$\LL$, we have
$$\LL R_{\left|_{Y_{p,s+2m-1}}\right.} = \mathrm{Id}_{Y_{p,s+2m-1}} +K,$$ 
with $K$ a compact operator, it follows that $\ran \left(\LL R_{p,s}\right)$ is finitely complemented.
Since 
$$\LL R(Y_{p,s+2m-1})\subset \LL(X_{p,s}), $$ 
it follows that the range of 
 $\LL$ is also finitely complemented, hence closed in $Y_{p,s+2m-1}$. 
 Finally, because $\ran P$ is finite dimensional, $\ran \left(\LL\right )+ \ran P$ is closed in $Y_{p,s+2m-1}$
 and $\LLs(X_{p,s}^{\sharp})$ is closed in $Y_{p,s+2m-1}^{\sharp}$.
\end{proof}
%
%
\subsection{Uniqueness}
\label{Uniqueness}
We begin by providing a uniqueness result for boundary layer potential solutions to (\ref{Dirichlet}) in the bounded domain $\Omega$,
and a partial result for $\reals^d\setminus \overline{\Omega}$.
Let us introduce the bilinear form $\cB$ via
$$\cB(w,v) = \begin{cases}
 \int_{\Omega} \langle \nabla \Delta^{(m-1)/2} w, \nabla \Delta^{(m-1)/2} v\rangle \dif x& m\text{ is odd},\\
  \int_{\Omega}  \Delta^{m/2} w(x) \Delta^{m/2} v(x) \dif x& m\text{ is even.}\end{cases}$$

The result for $\Omega$ we include for completeness; 
the fact that the solution has the specific form (\ref{BhRep}) is not important in this case. 
The result would follow easily from the classical uniqueness results in \cite{ADN,LionsMagenes}.
\begin{lemma}\label{inner_uniqueness}
Suppose that $u\in C^{2m}(\Omega)\cap C^{m-1}(\overline{\Omega})$ is a classical solution of (\ref{Dirichlet}) 
in $\Omega$ with homogeneous Dirichlet values (i.e., $h_k =0$ for all $k=0\dots m-1$).  
Then $u=0$ in $\Omega$. 
\end{lemma}
\begin{proof}
Recall Green's first identity 
$$\int_{\Omega} v(x)\Delta^m w(x) \dif x - \cB(w,v)  
= 
\sum_{j=0}^{m-1} (-1)^j \int_{\partial \Omega} \lambda_j v(x) \lambda _{2m-j-1} w(x)\, \dif \sigma(x).$$
Apply this to $w = u$ and $v = u$ and observe that $\cB(u,u) = 0$.

When $m$ is even, we see that  $\Delta^{m/2} u$ vanishes a.e. in $\Omega$.  
When $m$ is odd, $\Delta^{(m-1)/2}u$ must be a constant a.e. in $\Omega$, but since it satisfies
 $\lambda_{m-1}u = 0$,  it vanishes. In either case we have that $u$ satisfies the corresponding
polyharmonic Dirichlet problem of  order $\lfloor m/2\rfloor$.
Repeating this argument a maximum of $\log_2 m$ times, we arrive at $|\nabla u(x)| = 0$
throughout $\Omega$ and $ u|_{\partial \Omega} =0$.
\end{proof}

The corresponding problem for the exterior is more difficult. 
In general, uniqueness does not hold. 
E.g., both $u_1(x,y)= 1+\ln\bigl( x^2+y^2 \bigr)$ and $u_2(x,y)= x^2+y^2$  satisfy $\Delta^2 u =0$ in the complement of the
unit ball and they have the same trace and normal derivative on the unit circle. 
To treat this, we make use of a {\em radiating condition}, which guarantees uniqueness for functions having controlled growth.
In other words, under some additional assumptions of behavior of the function at infinity, 
the solution we propose will be in a unicity class for the unbounded domain
-- cf. \cite{CD-Radiating}. 

For the purposes of this article, we make them as follows:
for sufficiently large $R$ (relative to $\Omega$), 
consider  on $B(0,R)$ the boundary operators  $\lambda_j$ given in (\ref{bdry_operator_def}). 
The function $u$ satisfies the radiating conditions if  there exists $C$ so that
\begin{equation}\label{radiating}
|\lambda_{j}   u(x) |  \le C \begin{cases}R^{m-1-j},&\text{ for }j=0\dots m-1,\\
R^{m-d-j} ,& \text{ for }j=m\dots 2m-1.\end{cases}\end{equation}

\begin{lemma}\label{outer_uniqueness}
Suppose that $u\in C^{2m}(\reals^d\setminus\overline{\Omega})\cap C^{m-1}(\reals^d\setminus{\Omega})$ is a classical solution of (\ref{Dirichlet}) in 
$\reals^d\setminus\overline{\Omega}$ with homogeneous Dirichlet values.
 If $u$ satisfies (\ref{radiating}) as $R\to \infty$, then
 $\Delta^{m/2}u=0$ in $\reals^d\setminus\overline{\Omega}$, if $m$ is even, and
$\nabla \Delta^{(m-1)/2}u =0$  in $\reals^d\setminus\overline{\Omega}$, if $m$ is odd.
\end{lemma}
  \begin{proof}
 Considering Green's first identity in the set $\Upsilon_R = B(0,R)\setminus \overline{\Omega}$ 
 (for sufficiently 
large $R$) 
we have 
\begin{eqnarray*}
\int_{\Upsilon_R} u(x)\Delta^m u(x) \dif x - \cB(u,u)  
&=& 
\sum_{j=0}^{\lfloor m/2\rfloor} (-1)^j \int_{\partial \Upsilon_R} \lambda_ju(x) \lambda _{2m-j-1} u(x)\, \dif \sigma(x)
\end{eqnarray*}
By the homogenous Dirichlet conditions, the boundary integrals over $\partial\Omega$ vanish. 
This leaves
\begin{multline*}\sum_{j=0}^{m-1} (-1)^j \int_{\{|x|=R\}} \lambda_j 
u(x)
 \lambda _{2m-j-1} u(x)\, \dif \sigma(x)\\
\le
C
\sum_{j=0}^{m-1} \int_{\{|x|=R\}} R^{m-1-j}R^{1+j-d-m}   \dif \sigma(x) \\
\le
C
R^{d-1} R^{-d}    
\xrightarrow{R\to \infty} 0.
\end{multline*}
From this it follows that $\cB(u,u)=0$.
  \end{proof}

The radiating condition follows from the fact that the functions $\bfg$ are forced to lie in the kernel of $P^{\t}$.\footnote{This
should be a familiar phenomenon for practitioners of RBF interpolation: for scattered data fitting with conditionally positive definite functions,
 the interpolation matrix $(\phi(\xi-\zeta))$  is augmented by various polynomial side conditions 
 (and simultaneously, the addition of a polynomial, to keep the system square). 
 This has the dual effect of ensuring the interpolant lies in a {\em Native space} (a reproducing kernel semi-Hilbert space) and
 that the augmented system is $1-1$.}

\begin{proof}[Proof of Lemma \ref{oneone}]
Fix $p,s\in \reals$ and consider a solution $v=(A,\bfg) \in X_{p,s}^{\sharp}$ to the 
homogeneous system $\LLs v = 0$.
This implies that $\LL \bfg \in \Pi_{m-1}$, and
the ellipticity of $\LL$ -- in particular, the fact
that $\mathrm{sing}\supp\bfg \subset\mathrm{sing}\supp{ \LL \bfg} = \emptyset$ --  guarantees
that the entries $g_j$ of $\bfg$ are in $C^{\infty}(\partial \Omega)$. Because of this, and 
Corollary \ref{boundary_smoothness},
we can extend each
$V_j g_j$  to $C^{\infty}(\overline{\Omega})$ as well as $C^{\infty}(\reals^d\setminus \Omega)$, 
hence  the same holds for boundary layer potential 
$ u =  
\sum V_j g_j  + \sum A_j p_j$. 
 Thus $u$ satisfies (\ref{Dirichlet}) with homogeneous Dirichlet boundary conditions
 in both components of $\reals^d\setminus \partial\Omega$.

By Lemma \ref{inner_uniqueness}, $u=0$ in $\Omega$.

 To handle $u$ in the exterior of $\Omega$,  write $u = \phi* \mu_{\bfg} +p$, with
$\mu_{\bfg} = \sum_{j=0}^{m-1} \Lambda_j^{\t} (g_j\cdot\delta_{\partial \Omega})$.
Applying the moment conditions $P^{\t} \bfg = 0$, we see that
 $\mu_{\bfg} 
 \perp \Pi_{m-1}$. 
Thus, for any $q\in \Pi_{m-1}$, we have that 
$
\phi *\mu_{\bfg} =  (\phi - q) *\mu_{\bfg}$.
For $x$ sufficiently far from
$\partial \Omega$, let $Q_x$ be the Taylor polynomial of degree $m-1$ to $\alpha \mapsto \phi(x-\alpha)$ 
centered at the origin (or any other point suitably close to $\partial \Omega$). Then 
$$
(\phi - Q_x) *\mu_{\bfg}  = 
\sum_{j=0}^{m-1} 
\int_{\partial \Omega} \lambda_{j,\alpha} \left[\phi(x-\alpha) - Q_x(\alpha)\right] g_j(\alpha) 
\dif \alpha.
$$
Since $\lambda_j Q_x$ is the degree $m-1-j$ Taylor polynomial to $\alpha\mapsto \lambda_{j,\alpha}\phi(x-\alpha)$,
we have that 
$$| \lambda_{j,\alpha} \left[\phi(x-\alpha) - Q_x(\alpha)\right]| \le 
C
(\mathrm{diam}(\Omega))^{m-j}
\sup_{|\alpha|\le m} \max_{\alpha\in\Omega} |D^{\alpha}\phi(x -\alpha)|$$
From the remainder formula in Taylor's theorem and estimates on the derivative of the fundamental solution
(\ref{FS_derivative_estimates}), we have that
$
\phi *\mu_{\bfg} 
= \mathcal{O}(|x|^{m-d}|\log(x)|)$ as $|x|\to \infty$. 
Repeating this for derivatives of 
$ 
\phi *\mu_{\bfg} $, we observe that for $|\alpha|+d\le m$, we have
$
|D^{\alpha}   (\phi *\mu_{\bfg})(x) | 
 = \mathcal{O}(|x|^{m-d-|\alpha|}|\log x|)$
while for $|\alpha|+d>m$, we have
$
|D^{\alpha}(\phi *\mu_{\bfg})(x) |
= \mathcal{O}(|x|^{m-d-|\alpha|}).$
Because $D^{\alpha} p (x)= \mathcal{O}(|x|^{m-1-|\alpha|})$ for $\alpha \le m-1$,
the radiating conditions (\ref{radiating}) are satisfied by $u$, and Lemma \ref{outer_uniqueness} applies.

Because $\cB(u,u) = 0$ in $ \reals^d\setminus \partial \Omega$, 
we have $\Lambda_{m} u(x)=0$ for $x\in \reals^d\setminus \partial \Omega$. 
On the other hand, 
$$\Lambda_{m} u = \Lambda_{m} (\sum_{j=0}^{m-1} V_j g_j +p)
= \Lambda_{m}V_{m-1}g_{m-1}+  \sum_{j=0}^{m-1} \Lambda_{m}V_j g_j.
$$
By Corollary  \ref{jump_condition}, the sum $ \sum_{j=0}^{m-1} \Lambda_{m}V_j g_j$ is continuous throughout $\reals^d$,
while for $x_0\in \partial \Omega$,
\begin{eqnarray*}
\lim_{\substack{x\to x_0\\x\in \Omega}} \Lambda_{m}V_{m-1}g_{m-1}(x) -
\lim_{\substack{x\to x_0\\x\in \reals^d\setminus\overline{\Omega}}} \Lambda_{m}V_{m-1}g_{m-1}(x) 
= g_{m-1}(x_0)
\end{eqnarray*}
Therefore, 
$$\lim_{\substack{x\to x_0\\x\in \Omega}}\Lambda_{m} u(x) - \lim_{\substack{x\to x_0\\x\in\reals^d\setminus\overline{\Omega} }}  
\Lambda_{m} u(x) = g_{m-1}(x_0)  = 0.$$
The remaining auxiliary functions $g_{m-2},g_{m-3},\dots g_0$ 
can be treated 
using the same argument, with the operators $\Lambda_{m+1},\Lambda_{m+2},\dots,\Lambda_{2m-1}$.
Finally, it follows obviously that $p = u  \in \Pi_{m-1}$ vanishes.
\end{proof}
\section{Proofs of the main results}\label{proofs}
%
%
%
%
\begin{proof}[Proof of Theorem \ref{polyharmonic_dirichlet_solution}]
For $1<p<\infty$, the inverse of $L_{p,s}$ ensured by Proposition \ref{bounded_inverse} gives the 
auxiliary functions which are the solutions of the boundary potential solution of the Dirichlet problem.

Extend $\bfh \in Y_{p,s}$ to get $(0,\dots,0, h_0,\dots h_{m-1})^{\t} \in Y_{p,s}^{\sharp}$.
The solution is given by 
$$(A_1,\dots A_N, g_0,\dots,g_{m-1})^{\t} = L_{p,s}^{-1}(0,\dots,0, h_0,\dots h_{m-1})^{\t},$$
and one readily observes that $\bfg \in X_{p,s+1-2m}$
and
$\|\bfg\|_{X_{p,s+1-2m}}\le \|L_{p,s}^{-1}\| \|\bfh \|_{Y_{p,s}}$.
As desired,
 $$\|g_j\|_{W_p^{s+j+1-2m}(\partial \Omega)} \le \|L_{p,s}^{-1}\| 
\max\left(\|h_0\|_{H_p^s(\partial \Omega)}, \dots, \|h_{m-1}\|_{H_p^{s-(m-1)}(\partial \Omega)}\right)$$
 for each $j=0\dots m-1$.

\end{proof}
%
\begin{proof}[Proof of Corollary \ref{Dirichlet_corollary}]
Since $\bfh \in Y_{2,m-1/2}$,  $(0, \bfh)^{\t} \in Y_{2,m-1/2}^{\sharp}$,
and $(\vec{A},\bfg) = L^{-1}(0,\bfh)$ is in $X_{2,-m +1/2}^{\sharp}$. 
Consider the putative solution 
$$u=p+\sum_{j=0}^{m-1} V_j g_j= p+\phi*[(\Lambda_j)^{\t} (g_j \cdot \delta_{\partial \Omega})].$$
Note that for $j=0\dots m-1$, the functional $g_j \cdot \delta_{\partial \Omega}$ is in the dual of $H_2^{m-j}(\reals^d)$: 
 for any $f\in H_2^{m-j}( \reals^d)$, the trace theorem guarantees $\mathrm{Tr}f\in H_2^{m-j-1/2}(\partial \Omega)$. 
By definition,
$|\langle g_j \cdot \delta_{\partial \Omega}, f\rangle| = |\langle g_j, \mathrm{Tr}f\rangle |$, and thus
\begin{eqnarray*}
|\langle g_j \cdot \delta_{\partial \Omega}, f\rangle|
&\le& \|g_j\|_{H_2^{j-m+1/2}(\partial \Omega)} \| \mathrm{Tr}f\|_{H_2^{m-j-1/2}(\partial \Omega)}\\
&\le&
C \|g_j\|_{H_2^{j-m+1/2}(\partial \Omega)} \| f\|_{H_2^{m-j}( \reals^d)}.
\end{eqnarray*}
Thus,  $g_j \cdot \delta_{\partial \Omega} \in H_2^{-m +j}(\reals^d)$
and $\|g_j  \cdot \delta_{\partial \Omega} \|_{H_2^{-m +j}(\reals^d)} \le C \|g_j\|_{H_2^{-m + j+1/2}(\partial \Omega)}$
for each $j$. 

It follows that $(\Lambda_j)^{\t} (g_j \cdot \delta_{\partial \Omega})\in H_2^{-m}(\reals^d)$, 
and that (because $f\mapsto f*\phi$ is, up to a smoothing operator, a pseudodifferential operator of order $-2m$)
$\sum_{j=0}^{m-1}\phi*[(\Lambda_j)^{\t} (g_j \cdot \delta_{\partial \Omega})]\in H_{2,loc}^m(\reals^d)$. 
In short, we have 
$$\|u\|_{H_{2}^m(\Omega)}\le 
 C \max_{j=0\dots m-1} \|g_j \|_{H_2^{j-m+1/2}(\reals^d)}\le C \max_{k=0\dots m-1} \|h_k \|_{H_2^{m-1/2-k}(\reals^d)}.$$

The fact that $\Delta^m u  = 0 $ in $\Omega$ is  clear from the construction.
That the Dirichlet conditions are satisfied follows from a limiting argument: 
by Theorem \ref{polyharmonic_dirichlet_solution} the conditions hold for $\bfh \in (C^{\infty})^m$;
this extends to $\bfh \in Y_{2,m-1/2}$ by density of $(C^{\infty})^m$ and the
continuity of the map  $\bfh \mapsto u$ given by $\|u\|_{H_2^m(\Omega)} \le C \|(\vec{A},\bfg)\|_{X_{2,-m+1/2}^{\sharp}}\le C \|\bfh\|_{Y_{2,m-1/2}}$.
\end{proof}
%
\begin{proof}[Proof of Theorem \ref{Main_Rep_Theorem}]
We begin by considering $f\in C^{\infty}(\overline{\Omega})$. 
In this case, Theorem \ref{polyharmonic_dirichlet_solution}
ensures that there are $C^{\infty}$ functions $g_j$, $j=0,\dots,m-1$ 
and a polynomial $p\in\Pi_{m-1}$
so that $f_1= \sum_{j=0}^{m-1} V_j g_j +p$
solves the polyharmonic Dirichlet problem (\ref{Dirichlet})
with boundary data $h_k= \lambda_k f \in C^{\infty}(\partial \Omega)$, for $k=0\dots m-1$.

By Proposition \ref{boundary_smoothness}, the remainder  $f_2 := f-f_1$ is in $C^{\infty}(\overline{\Omega})$ and
Green's representation (\ref{GreensRep}) gives
\begin{eqnarray*}
f_2(x) 
&=&
\int_{\Omega} \Delta^m f(\alpha) \phi(x-\alpha) \dif\alpha\\
&&+ \sum_{j=0}^{m-1} (-1)^{j+1} \int_{\partial \Omega} (\lambda_{2m-j-1} (f-f_1))(\alpha) \; \lambda_{j,\alpha}\phi(x-\alpha)\, \dif \sigma(\alpha).
\end{eqnarray*}
The representation of $f$ follows, and we note that 
\begin{equation}\label{enn_jay}
N_j f =  g_j + (-1)^{j+1}(\lambda_{2m-j-1} f - \lambda_{2m-j-1}f_1)
\end{equation}
holds.

For every $s\ge 0$ 
there is $C<\infty$
so that
$\|\lambda_k f \|_{H_p^s(\partial \Omega)}\le C \|f\|_{B_{p,1}^{s+k+1/p}(\Omega)}$ 
by the trace theorem, specifically the fact that
 $\mathrm{Tr}: B_{p,1}^{1/p}(\Omega)\to L_p(\partial \Omega)$ is bounded (this is  \cite[Section 4.4.3]{Trieb}). 
In particular
\begin{equation}\label{first_p_estimate}
\|\lambda_{2m-j-1} f \|_{H_p^s(\partial \Omega)} \le C\|f\|_{B_{p,1}^{s+2m-j-1+1/p}(\Omega)}
\end{equation}
holds for all $s\ge 0$ and all integers $0\le j\le 2m-1$.
It follows from Theorem \ref{polyharmonic_dirichlet_solution} and (\ref{first_p_estimate}), 
that for $k = 0, \dots, m-1$,
\begin{equation}\label{second_p_estimate}
\|g_k\|_{H_p^s(\partial \Omega)}
\le 
C\max_{j=0\dots m-1} \|\lambda_{j} f\|_{H_p^{s+(2m-j-k-1)}(\partial \Omega)}\le
 C \|f\|_{B_{p,1}^{s+2m-k-1+1/p}(\Omega)}.
\end{equation}
It remains to consider $\lambda_{2m-j-1}f_1 = \lambda_{2m-j-1}(\sum_{k=0}^{m-1} V_k g_k +p)$. 
Employing the boundary operators, this  
simplifies to $\sum_{k=0}^{m-1} v_{2m-j-1,k}^- g_k$,
since $p\in \Pi_{m-1}$ and $j\le m-1$.
It follows from Lemma \ref{smoothing_lemma} that
\begin{equation*}
\|\lambda_{2m-j-1}f_1 \|_{H_p^s(\partial \Omega)} 
\le
\sum_{k=0}^{m-1} \|v_{2m-j-1,k}^- g_k\|_{H_p^s(\partial\Omega)}\\
 \le
 \sum_{k=0}^{m-1}\|g_k\|_{H_p^{(s+k-j)}(\partial\Omega)}
 \end{equation*}
holds.
We  use (\ref{second_p_estimate}),  
namely $\|g_k\|_{H_p^{(s+k-j)}(\partial\Omega)}\le C  \|f\|_{B_{p,1}^{s+2m-j-1+1/p}(\Omega)}$,
to establish
\begin{equation}\label{third_p_estimate}
\|\lambda_{2m-j-1}f_1 \|_{H_p^s(\partial \Omega)} 
\le C\,\|f\|_{B_{p,1}^{s+2m-j-1+1/p}(\Omega)}.
\end{equation}
Applying the triangle inequality to (\ref{enn_jay}) gives
\begin{eqnarray*}
\|N_j f\|_{H_p^{s}(\partial \Omega)}
&\le &
 \|g_j\|_{H_p^{s}(\partial \Omega)} + \|\lambda_{2m-j-1} f\|_{H_p^{s}(\partial \Omega)} +\| \lambda_{2m-j-1}f_1\|_{H_p^{s}(\partial \Omega)}\\
&\le& C\|f\|_{B_{p,1}^{s+2m-j-1}(\Omega)},
\end{eqnarray*}
where the final inequality follows from the three estimates  (\ref{second_p_estimate}), (\ref{first_p_estimate}) and (\ref{third_p_estimate}). 

For a general $f\in W_p^{2m}(\Omega)$, the representation (\ref{Main_Rep}) holds by density of $C^{\infty}(\overline{\Omega})$ 
and continuity of the operators $\Delta^{m}$ and $N_j$, $j=0\dots m-1$.

\end{proof}

 \section{Surface spline approximation}\label{S:SSA}
\subsection{Approximation scheme}
We now develop the approximation scheme based on the integral identity introduced in Section \ref{S:ID}. 
The scheme and the accompanying error estimate
are generalizations of the scheme in \cite{Hdisk}. 
Specifically, the  approximation scheme takes the form
\begin{equation*}\label{4E:scheme}
T_{\Xi} f (x) =
\int_{\Omega} \Delta^m f(\alpha) k(x,\alpha)\, \dif\alpha
\mbox{}+\sum_{j=0}^{m-1}\int_{\partial \Omega} N_jf(\alpha) k_j(x,\alpha)\, \dif\sigma(\alpha) 
+p(x)
\end{equation*}
with $g_j =N_jf$ the 
auxiliary terms  and $p$ the polynomial from Theorem \ref{Main_Rep_Theorem}. 
The challenge is to
find suitable {\em replacement kernels}  $k (x,\alpha)  = \sum_{\xi}a(\alpha,\xi) \phi(x-\xi)$ and $k_j(x,\alpha) = \sum_{\xi}a_j(\alpha,\xi) \phi(x-\xi)$ so that the {\em error kernels}
\begin{eqnarray}
E(x,\alpha)&:=&|k(x,\alpha)-\phi(x-\alpha)|,\label{4errker1}\\
E_j(x,\alpha)&:=&|k_j(x,\alpha)-\lambda_{j,\alpha}\phi(x-\alpha)|,\ j = 0\dots m-1\label{4errker2}
\end{eqnarray}
are uniformly small and decay rapidly as $|x-\alpha|\to \infty$.

It follows that the (pointwise) error incurred from the approximation scheme can be estimated (for sufficiently smooth $f$) by
$$
|f(x) - T_{\Xi} f(x)|\le 
\int_{\Omega}
E(x,\alpha)
 |\Delta^m f(\alpha) |
\, \dif\alpha
+
\sum_{j=0}^{m-1}\int_{\partial \Omega} E_j(x,\alpha) |N_jf(\alpha)| \, \dif\sigma(\alpha). 
$$
From this it is clear that for $f\in W_p^{2m}(\Omega)$, the $L_p$ error is bounded by
\begin{equation}\label{Approx_Lp_error_operator_norms}
\|f - T_{\Xi} \|_{L_p(\Omega)}
\le 
\|\mathcal{E}\|_{p\to p}
 \|\Delta^m f |\|_{L_p(\Omega)}
+
\sum_{j=0}^{m-1}\| \mathcal{E}_j\|_{p\to p} \|N_j f\|_{L_p(\partial \Omega)}. 
\end{equation}
where $\mathcal{E}$ and $\mathcal{E}_j$ are the integral operators induced by the error kernels $E$ and $E_j$.

In the next subsection, we state the conditions on the centers necessary for a high rate of convergence, 
how these conditions imply uniform bounds and rapid decay of the error kernels -- 
measured in terms of a density parameter introduced in the conditions on the centers -- 
and we demonstrate that the fill distance $h$, 
cf. (\ref{fill_dist}), is connected to this density parameter. 
The following subsection states the main theorem and its consequences, 
including a corollary showing that an increase in density of centers 
near the boundary yields the optimal (boundary-free) approximation order.

\subsection{Error kernels}\label{SS:Error_kernels}
In this section we describe how to construct replacement kernels $k$ and $k_j$, $j=0\dots m-1$
and give pointwise estimates for the corresponding 
error kernels $E$ and $E_j$, $j=0\dots m-1$.
Following this, we give operator norms for the integral operators defined by $E$ and $E_j$, which leads to
components of the approximation error in (\ref{Approx_Lp_error_operator_norms}).

\paragraph{Interior kernel}
To construct $k(x,\alpha) = \sum_{\xi\in\Xi} a(\alpha,\xi)\phi(x-\xi)$, we require the properties of the 
coefficient kernel $(\alpha,\xi) \mapsto a(\alpha,\xi)$ of the following type.
 For $K,R>0$, and $M\in\nats$ 
 \begin{enumerate}
\item $\max_{\alpha\in \overline{\Omega}} \sum_{\xi\in\Xi} |a(\alpha,\xi)| \le K$
\item  For every $p\in \Pi_{M}$, $\sum_{\xi\in\Xi} a(\alpha,\xi) p (\xi) = p(\alpha)$.
\item If $|\alpha - \xi| >R$ then $a(\alpha,\xi) = 0$.
\end{enumerate}
We call such a coefficient kernel a stable, local polynomial
reproduction of order $M$, radius $R$ and stability $K$. Local polynomial reproductions
have a long history in RBF approximation and  related fields -- 
one may see their use in \cite{WuSch,JSW,Wend,DeRo,Hdisk}, for example.

We now recall a  result guaranteeing that such local polynomial reproductions exist for  regions with Lipschitz boundary 
satisfying an interior cone condition\footnote{a much weaker condition on $\Omega$ than we assume in this article}
with aperture given by the angle $\theta$ 
and radius $r$ -- namely, for regions $\Omega$ having the property that for every $\alpha \in \overline{\Omega}$ 
there is $\nu_{\alpha}$ so that the  cone
$$C(\alpha, r, \theta, \nu_{\alpha}): = 
\left\{x\in \reals^d\mid |x-\alpha|\le r,
\left\langle\frac{x-\alpha}{|x-\alpha|},\nu_{\alpha}\right\rangle
\ge
\cos \alpha\right\} 
$$
is contained in $\overline{\Omega}$.

The result we cite is  the so-called {\em norming set} result \cite[Theorem 3.14]{Wend}, which  
ensures that
for every 
$M\in \nats$, and $\Xi$ sufficiently dense (with $h = \max_{x\in\Omega}\dist(x,\Xi)$ sufficiently small -- 
i.e., bounded above by a constant depending on  $\Omega$ and $M$), appropriately rescaled cones $C$ contain subsets of $\Xi$ 
so that the  norm of a polynomial of degree $M$ over $C$ is controlled by its values on $\Xi\cap C$.
In short, there is a $\Gamma>0$ depending on the cone parameters $r,\theta$ of $\Omega$
so that
for every $p\in \Pi_M$ 
the uniform norm over the rescaled cone $C(\alpha) = C(\alpha,\Gamma M^2h, \theta, \nu_{\alpha})$ 
is controlled by the finite subset obtained from $\Xi$ (i.e., the norming set). Indeed,
\begin{equation}\label{norming_set}
\|p\|_{L_{\infty}(C(\alpha))}\le 2 \left\|p\left|_{\Xi\cap C(\alpha)}\right.\right\|_{\ell_{\infty}(\Xi\cap C(\alpha))}.
\end{equation}
Note that beside the requirement that $h$ is sufficiently small, the geometry of $\Xi$ does not play a role in this estimate.

Using (\ref{norming_set}), it  is possible to construct a functional $\mu_{\alpha}$ in the dual of $\ell_{\infty}(\Xi\cap C(\alpha))$ which represents  
$\delta_{\alpha} :p \mapsto p(\alpha)$. 
Namely, $\delta_{\alpha} p = \mu_{\alpha} (p\left|_{\Xi\cap C(\alpha)}\right.) =\sum_{\xi\in \Xi\cap C(\alpha)} a(\alpha,\xi) p(\xi)$
for some sequence $\bigl( a(\alpha,\xi) \bigr)_{\xi\in \Xi\cap C(\alpha)}\in \ell_1$.
 Because 
$\|\mu_{\alpha}\|_{\ell_{\infty}'} \le 2\|\delta_{\alpha}\|_{L_{\infty}'} \le 2$, we 
have $\sum_{\xi\in \Xi\cap C(\alpha)} |a(\alpha,\xi)| \le 2.$
We extend the sequence $ \bigl( a(\alpha,\xi) \bigr)_{\xi\in \Xi\cap C(\alpha)}$ by zero (i.e., $a(\alpha,\xi) =0$ for $\xi\notin C(\alpha)$) so that 
$\bigl( a(\alpha,\xi) \bigr)_{\xi\in \Xi }\in  \ell_1(\Xi)$

Let $M= 2m$ and use \cite[Theorem 3.14]{Wend} to generate the stable, local polynomial reproduction $a$. 
Then $k(x,\alpha)= \sum_{\xi\in\Xi} a(\alpha,\xi)\phi(x-\xi)$ is the replacement kernel. The error kernel satisfies, for every $x,\alpha\in \Omega$
$$E(x,\alpha) = |\phi(x-\alpha) - k(x,\alpha)| \le C h^{2m-d} \left(1+\frac{\dist(x,\alpha)}{h}\right)^{-(d+1)}$$ 
with a constant $C$ depending only on $M$ and the cone parameters $\theta$ and $\rho$.
In particular, $E$ satisfies
$$\max_{x\in\Omega}\int_{\Omega} E(x,\alpha) \dif \alpha \le C h^{2m},
\qquad \max_{\alpha\in\Omega}\int_{\Omega} E(x,\alpha) \dif x \le C h^{2m}.$$
It follows that for $1\le p\le \infty$, 
the integral operator $\mathcal{E}:f\mapsto \int_{\Omega}E(x,\alpha) f(\alpha) \dif \alpha$ is bounded, in $L_p$
like $\|\mathcal{E}\|_{p\to p} \le Ch^{2m}$ and for $f\in W_p^{2m}(\Omega)$
\begin{equation}\label{interior_error}
\left|\int_{\Omega}\Delta^m f(\alpha) \phi(x-\alpha) \dif \alpha -
 \sum_{\xi\in\Xi} A_{\xi} \phi(x-\xi) \right| \le 
 Ch^{2m}\|\Delta^m f\|_{L_p(\Omega)}.
 \end{equation}
 where $A_{\xi} :=\int_{\Omega} a(\alpha,\xi) \Delta^m f(\alpha)\dif \alpha$.
\paragraph{Boundary kernels}
To construct $k_j(x,\alpha) = \sum_{\xi\in\Xi} a_j(\alpha,\xi)\phi(x-\xi)$, we require similar properties of the 
coefficient kernels. For $K,R>0$, and $M\in\nats$  we seek $a_j$ so that
 \begin{enumerate}
\item  For every $p\in \Pi_{M}$, $\sum_{\xi\in\Xi} a(\alpha,\xi) p (\xi) = \lambda_jp(\alpha)$.
\item If $|\alpha - \xi| >R$ then $a_j(\alpha,\xi) = 0$.
\item $\max_{\alpha\in \partial \Omega} \sum_{\xi\in\Xi} |a_j(\alpha,\xi)| \le K h^{-j}$
\end{enumerate}

We can again use the norming set result (\ref{norming_set}) to build representers for the functionals 
$\delta_{\alpha} \Lambda_j$, which have norms given by Bernstein's inequality
$| \Lambda_j p(\alpha)| \le C_{\Omega} ( M^2/h)^{j} \|p\|_{L_{\infty}(C(\alpha))} $
 (Bernstein's inequality is given, for example, in \cite[Proposition 11.6]{Wend}).
  It follows that
there is  a representer $\mu_{j,\alpha}$ in the dual of $\ell_{\infty}(\Xi\cap C(\alpha))$ for the functional $p \mapsto \Lambda_j p(\alpha)$
in the sense that
$ \Lambda_j p(\alpha) = \mu_{j,\alpha} (p\left|_{\Xi\cap C(\alpha)}\right.) =\sum_{\xi\in \Xi\cap C(\alpha)} a_j(\alpha,\xi) p(\xi)$
where $\bigl(a_j(\alpha,\xi)\bigr)_{\xi\in\Xi\cap C(\alpha)}$ so that
$$
\sum_{\xi\in \Xi\cap C(\alpha)} |a_j(\alpha,\xi)| 
= 
\|\mu_{j,\alpha}\|
\le 
2 \|p\mapsto \Lambda_j p(\alpha)\| 
\le 
C_{\Omega}M^{2j} h^{-j}  
$$
(this is \cite[Theorem 11.8]{Wend}). We extend by zero so that $a_j(\alpha,\xi)= 0$ for  $\xi \notin C(\alpha)$.

The replacement kernels are given by $k_j(x,\alpha) = \sum_{\xi\in\Xi}a_j(\alpha,\xi) \phi(x-\xi)$. 
In this case, it suffices to take $M = 2m-j$.
The error kernel satisfies, for every $x,\alpha\in \Omega$
$$
E_j(x,\alpha) = |\lambda_{j,\alpha}\phi(x-\alpha) - k_j(x,\alpha)| 
\le 
C h^{2m-d-j} \left(1+\frac{\dist(x,\alpha)}{h}\right)^{-(d+1)}
$$ 
with a constant $C$ depending only on $M$ and the cone parameters $\theta$ and $\rho$.
In particular, $E_j$ satisfies
$$\max_{x\in\Omega}\int_{\partial \Omega} E_j(x,\alpha) \dif \sigma(\alpha) \le C h^{2m-j-1},
\qquad \max_{\alpha\in\partial \Omega}\int_{\Omega} E_j(x,\alpha) \dif x \le C h^{2m-j}.$$
It follows that for $1\le p\le \infty$, 
the  operator 
$\mathcal{E}_j:f\mapsto \int_{\partial \Omega}E_j(x,\alpha) f(\alpha) \dif \sigma(\alpha)$ 
is bounded, in $L_p$ like 
$\|\mathcal{E}_j\|_{p\to p} \le Ch^{2m-j-1+1/p}$ and for $f\in W_p^{2m}(\Omega)$
\begin{multline}\label{boundary_errors}
\left|\int_{\partial \Omega}N_j f(\alpha) \lambda_{j,\alpha} \phi(x-\alpha) \dif \sigma(\alpha) -
 \sum_{\xi\in\Xi} A_{j,\xi} \phi(x-\xi) \right| \\
 \le 
 Ch^{2m-j-1+1/p}\|N_j f\|_{L_p(\Omega)}.
 \end{multline}
 where $A_{j,\xi} := \int_{\partial\Omega} a_j(\alpha,\xi) N_j f(\alpha)\dif \sigma(\alpha)$.
 
 We are now in a position to give approximation rates for the operators $T_{\Xi}$  for functions of {\em full} smoothness. 
 %
 %
 \begin{lemma} \label{approximation_by_operator}
 Let $f\in W_p^{2m}(\Omega)$ (or $C^{2m}(\overline{\Omega})$ in case $p=\infty$). Then there are 
positive constants $h_0$ and $C$ (depending on $\Omega$ and $m$) so that for all $h\le h_0$,
 $$\|f - T_{\Xi} f\|_{L_p(\Omega)} \le C\left(h^{2m}  \|\Delta^m f\|_{L_p(\Omega)}
 +\sum_{j=0}^{m-1} h^{2m-j-1+\frac1{p}} \|N_j f\|_{L_p(\partial \Omega)}\right).$$
 \end{lemma}
 %
 \begin{proof}
 The lemma follows directly from Theorem \ref{Main_Rep_Theorem}, (\ref{interior_error}) and (\ref{boundary_errors}).
 \end{proof}
 %
 %
 \subsection{Approximation results}
 Our first result about surface spline approximation is broken into  three parts, treating approximation in $L_p$, with 
 $1<p<\infty$, treating approximation in $L_1$ and approximation in $L_{\infty}$. 
 The error estimates follow along the lines of \cite{Hdisk}. 
 
 In each case we make use of  a $K$-functional argument to allow the operator to handle functions of lower smoothness.
  Let $f_e$ be the universal extension to $\reals^d$ of the target function $f$ defined on $\Omega$
   guaranteed by 
 \cite[Theorem 2.2]{Rych}.
   Let $\eta:\reals^d\to [0,1]$ be a compactly supported $C^{\infty}$ function satisfying 
 $\int_{\reals^d}  x^{\alpha}\eta(x) \dif x= \delta(|\alpha|)$ for $|\alpha|\le 2m$. 
 This ensures that $\eta* p = p$ for  all $p\in \Pi_{2m}$.
 The $L_1$ preserving 
 dilation of $\eta$ by $h$ is $\eta_h := h^{-d} \eta(\cdot/h)$.
 We define  $S_h f\in W_p^{2m}(\reals^d)$ 
 as $S_h f = f_e * \eta_h\in C^{\infty}(\reals^d)$.
 
 In short, for a function having smoothness $m+1/p$ (this is made precise below) we have that 
 $S_h f\in C^{\infty}(\Omega)$. Consequently,
 \begin{itemize}
\item $\|f-S_h f\|_p=\mathcal{O}( h^{m+1/p} )$;
\item $\|S_h f\|_{W_p^{2m}(\Omega)} = \mathcal{O}(h^{1/p-m})$;
\item for every $0<s<2m$, $\|S_h f\|_{B_{p,1}^{s}(\Omega)} = \mathcal{O}(h^{m+1/p-s})$.
 \end{itemize}
 It follows that we can apply Lemma \ref{approximation_by_operator} to estimate 
 $\|S_h f-T_{\Xi} S_h f\|_{p}$. 
  We have
 that 
 $\|N_j S_h f\|_{L_p(\partial \Omega)} \le C \|S_h f\|_{B_{p,1}^{2m-j-1+1/p}( \Omega)}
 =\mathcal{O}(
 h^{m+1/p-(2m-j-1+1/p)} )= \mathcal{O}(h^{j+1-m})
 $
  by Theorem 
 \ref{Main_Rep_Theorem}.
Using this, we control the boundary error terms appearing in Lemma \ref{approximation_by_operator}:
 \begin{eqnarray*}
 h^{2m-j-1+1/p}\|N_j S_h f\|_{L_p(\partial \Omega)} 
& \le& 
 C h^{2m-j-1+1/p}\|S_h f\|_{B_{p,1}^{2m-j-1+1/p}( \Omega)}\\
& =&\mathcal{O}(h^{m+1/p}).
 \end{eqnarray*}
  \begin{theorem}[Approximation in $L_p$, $1<p<\infty$]
 Let  $1<p<\infty$ and suppose $f\in B_{p,1}^{m+1/p}(\Omega)$. There exist positive constants $C,h_0$ so that
 for every $\Xi\subset \Omega$ satisfying $h\le h_0$  there is $s_{f}\in S_{m-1}(\Xi)$ so that
 $$\|f-s_f\|_p \le h^{m+1/p} \|f\|_{B_{p,1}^{m+1/p}(\Omega)}.$$
 \end{theorem}
 \begin{proof}
 It is an easy exercise to demonstrate the inequalities
 \begin{eqnarray*}
 \|f_e -S_h f\|_{L_p(\reals^d)} &\le& C h^{m+1/p} \|f_e\|_{B_{p,1}^{m+1/p}(\reals^d)} \le C h^{m+1/p} \|f\|_{B_{p,1}^{m+1/p}(\Omega)},\\
 \|S_h f\|_{W_p^{2m}(\reals^d)} &\le& C h^{1/p-m} \|f_e\|_{B_{p,1}^{m+1/p}(\reals^d)} \le C h^{1/p-m} \|f_e\|_{B_{p,1}^{m+1/p}(\Omega)}
 \end{eqnarray*}
 and, for $0<s<2m$, 
  $$
  \|S_h f\|_{B_{p,1}^{s}(\reals^d)} 
  \le C h^{m+1/p-s} \|f_e\|_{B_{p,1}^{m+1/p}(\reals^d)} \le 
  C h^{m+1/p-s} \|f\|_{B_{p,1}^{m+1/p}(\Omega)}.
  $$
  
By the second inequality, we have   
  $
  \|\Delta^m S_h f\|_{L_p(\Omega)}\le C h^{1/p-m} \|f\|_{B_{p,1}^{m+1/p}(\Omega)}
  $. Theorem \ref{Main_Rep_Theorem} and the third estimate imply that
  $$
  \|N_j S_h f\|_{L_p(\partial \Omega)} 
  \le   
C  \|S_h f\|_{B_{p,1}^{2m-j-1+1/p}(\Omega)}
  \le 
  C h^{j+1-m} \|f\|_{B_{p,1}^{m+1/p}(\Omega)}.
  $$ 
  We   now   apply $T_{\Xi}$ to $S_h f\left|_{\Omega}\right.$ (the restriction of $S_h f$ to $\Omega$).
By %
Lemma \ref{approximation_by_operator}, we see that 
 $\| S_h f\left|_{\Omega}\right.- T_{\Xi}(S_h f\left|_{\Omega}\right.)\|_{L_p(\Omega)} \le Ch^{2m} h^{1/p-m} \|f_e\|_{B_{p,1}^{m+1/p}(\Omega)}$ and 
 \begin{equation}
 \|f - T_{\Xi}(S_h f\left|_{\Omega})\right.\|_{L_p(\Omega)} 
 \le
 C h^{m+1/p} \|f\|_{B_{p,1}^{m+1/p}(\Omega)}   
  \end{equation}
 from which the theorem follows.
\end{proof}
%
In order to get a result  result for  approximation in $L_1$, we assume slightly more more smoothness for the target function 
\footnote{This is because of challenges in bounding pseudo-differential operators $N_j$  on spaces measuring smoothness in $L_1$. 
Although it may appear to be an artifact of working in this setting  
(after all, there are many pseudo-differential operators that do not have such difficulties: constant coefficient differential operators, for instance) 
in the case where $\Omega$ is the disk in 
$\reals^2$, it is known that the operators $N_j$ are not bounded from $W_1^{m+1}$ to $L_1$.}

 \begin{theorem}[Approximation in $L_1$]
 Let  $\epsilon>0$ and suppose $f\in B_{1,1}^{m+1+\epsilon}(\Omega)$. 
 There exist positive constants $C_{\epsilon},h_0$ so that
 for every $\Xi\subset \Omega$ satisfying $h\le h_0$  there is $s_{f}\in S_{m-1}(\Xi)$ so that
 $$\|f-s_f\|_1 \le C_{\epsilon} h^{m+1/p} \|f\|_{B_{1,1}^{m+1/p}(\Omega)}.$$
 \end{theorem}
 \begin{proof}
As in the previous case, one easily demonstrate the inequalities
 \begin{eqnarray*}
 \|f_e - S_h f\|_{L_1(\reals^d)} &\le& C h^{m+1} \|f_e\|_{B_{1,1}^{m+1+\epsilon}(\reals^d)} \le C h^{m+1} \|f\|_{B_{1,1}^{m+1+\epsilon}(\Omega)},\\
 \|S_h f\|_{W_1^{2m}(\reals^d)} &\le& C h^{1-m} \|f_e\|_{B_{1,1}^{m+1}(\reals^d)} \le C h^{1-m} \|f_e\|_{B_{1,1}^{m+1}(\Omega)}
 \end{eqnarray*}
 and, for $0<s<2m$ and $\epsilon>0$
  $$
  \|S_h f\|_{B_{1,1}^{s+\epsilon}(\reals^d)} 
  \le C h^{m+1-s} \|f_e\|_{B_{1,1}^{m+1+\epsilon}(\reals^d)} \le 
  C h^{m+1-s} \|f\|_{B_{1,1}^{m+1+\epsilon}(\Omega)}.
  $$
Let $p = \frac{d-1}{d-1-\epsilon}$. 
 By compactness, we have $\|N_j S_h f\|_{L_1(\partial \Omega)} \le C \|N_j S_h f\|_{L_p(\partial \Omega)} $.
 By  Theorem \ref{Main_Rep_Theorem}, 
$ \|N_j S_h f\|_{L_p(\partial \Omega)}\le C \|S_h f\|_{B_{p,1}^{2m-j-1+1/p}( \Omega)}$.
Finally, an application of the Sobolev   embedding theorem for Besov spaces, \cite[Theorem 2.7.1]{Trieb1}, we have 
$\|S_h f\|_{B_{p,1}^{2m-j-1+1/p}( \Omega)}\le C \|S_h f\|_{B_{1,1}^{2m-j-1+\epsilon}( \Omega)}$.
Together,
  we obtain 
 $$
 \|N_j S_h f\|_{L_1(\partial \Omega)} 
 \le 
 C \|S_h f\|_{B_{1,1}^{2m-j-1+\epsilon}( \Omega)}
   \le 
  C h^{j+1-m} \|f\|_{B_{1,1}^{m+1+\epsilon}(\Omega)}.
 $$ 
 
 Applying $T_{\Xi}$ to the restriction $S_h f\left|_{\Omega}\right.$, we see that 
 $$\| S_h f\left|_{\Omega}\right.- T_{\Xi}(S_h f\left|_{\Omega}\right.)\|_{L_1(\Omega)} \le Ch^{2m} h^{1-m} \|f_e\|_{B_{1,\infty}^{m+1}(\Omega)}$$ 
 and 
 \begin{equation}
 \|f - T_{\Xi}(S_hf\left|_{\Omega})\right.\|_{L_1(\Omega)} 
 \le
 C h^{m+1} \|f\|_{B_{1,\infty}^{m+1+\epsilon}(\Omega)}   
  \end{equation}
 from which the theorem follows.
\end{proof}

The final case follows the same lines. 
 \begin{theorem}[Approximation in $L_{\infty}$]
 Let  $\epsilon>0$ and suppose $f\in C^{m+\epsilon}(\overline{\Omega})$. 
 There exist positive constants $C,h_0$ so that
 for every $\Xi\subset \Omega$ satisfying $h\le h_0$  there is $s_{f}\in S_{m-1}(\Xi)$ so that
 $$\|f-s_f\|_{\infty} \le h^{m} \|f\|_{C^{m+\epsilon}(\overline{\Omega})}.$$
 \end{theorem}
 \begin{proof}
In this case, we have
 \begin{eqnarray*}
 \|f_e - S_h f\|_{L_{\infty}(\reals^d)} &\le& C h^{m} \|f_e\|_{C^{m}(\reals^d)} \le C h^{m} \|f\|_{C^{m}(\overline{\Omega})},\\
 \|S_h f\|_{C^{2m}(\reals^d)} &\le& C h^{-m} \|f_e\|_{C^{m}(\reals^d)} \le C h^{-m} \|f_e\|_{C^{m}(\overline{\Omega})}
 \end{eqnarray*}
 and, for $m<s<2m$, 
  $$
  \|S_h f\|_{C^{s}(\reals^d)} 
  \le C h^{m+\epsilon-s} \|f_e\|_{C^{m+\epsilon}(\reals^d)} \le 
  C h^{m+\epsilon-s} \|f\|_{C^{m+\epsilon}(\overline{\Omega})}.
  $$
  In this case, we apply Theorem \ref{Main_Rep_Theorem} to $S_hf$ as follows. 
  For any $2d/\epsilon<p<\epsilon$, the Sobolev embedding theorem guarantees the embedding 
  $H_p^{\epsilon/2}(\partial \Omega) \subset L_{\infty}(\partial \Omega)$. 
  Applying this to $N_j S_h f$ gives
   $
 \|N_j S_h f\|_{L_{\infty}(\partial \Omega)} 
 \le 
C  \|N_j S_h f\|_{H_p^{\epsilon/2}(\partial \Omega)} 
 $.
By Theorem \ref{Main_Rep_Theorem},  $ \|N_j S_h f\|_{L_{\infty}(\partial \Omega)} \le C \|f\|_{B_{p,1}^{2m-j-1+\epsilon/2}({\Omega})}$ holds.
Finally, we have the  embedding $C^{\sigma}(\overline{\Omega}) \subset B_{p,\infty}^{\sigma}(\Omega)$, valid for any $\sigma>0$ by the compactness
of $\Omega$. 
This gives $C^{2m-j-1+\epsilon}(\overline{\Omega})\subset B_{p,1}^{2m-j-1+\epsilon/2}({\Omega})$. 
and, therefore,
 $$
 \|N_j S_h f\|_{L_{\infty}(\partial \Omega)} 
 \le 
 C \|S_h f\|_{C^{2m-j-1+\epsilon}( \Omega)}.
 $$ 
In particular this holds for $s = 2m-j-1+\epsilon$, which satisfies $m<s<2m$.
 
  Applying $T_{\Xi}$ to the restriction $S_h f\left|_{\overline{\Omega}}\right.$, Lemma \ref{approximation_by_operator} and the above estimates ensure that 
 $$\| S_h f\left|_{\overline{\Omega}}\right.- T_{\Xi}(S_h f\left|_{\overline{\Omega}}\right.)\|_{L_{\infty}(\Omega)} \le Ch^{m} \|f_e\|_{C^{m+\epsilon}(\overline{\Omega})}$$ 
 and 
 \begin{equation}
 \|f - T_{\Xi}(S_hf\left|_{\overline{\Omega}})\right.\|_{L_{\infty}(\Omega)} 
 \le
 C h^{m} \|f\|_{C^{m+\epsilon}(\overline{\Omega})}   
  \end{equation}
 from which the theorem follows.
\end{proof}

\subsection{Overcoming boundary effects}
We now demonstrate that the ``free space'' approximation order of $2m$ can be attained by increasing the density of centers in a small neighborhood of the boundary. 
This approach was shown to be successful in \cite{Hdisk}, and is
similar to  quadratic oversampling used by Rieger and Zwicknagl \cite{RZ}.
We now describe how to modify the error estimate from Lemma \ref{approximation_by_operator}, 
to respond to oversampling near the boundary.

We add an extra assumption about $\Xi$ -- namely that the sampling density of $\Xi$ near the boundary is $h^{\nu}$ rather than $h$. 
In this case ``near'' means within a tube which has thickness $\propto h^{\nu}$.

To proceed, we fix an ``oversampling factor'' $\nu\ge1$. 
By the smoothness and compactness of the boundary,
$\Omega$ satisfies an interior cone condition.
Indeed, for every aperture $0\le \theta< \pi/2$
there is a  radius $r$ so that
for every $\alpha\in \partial \Omega$, 
the cone $C(\alpha, r,\theta,-\vec{n}_{\alpha})$
lies in  $\overline{\Omega}$.

It follows from \cite[Theorem 3.8]{Wend} 
that if 
$\Omega_{h,\nu} = \{\xi\in\Omega \mid \dist(\xi,\partial \Omega) \le 12 h^{\nu} m^2\}$
satisfies the estimate
$\max_{x\in \Omega_{h,\nu} }\dist(x,(\Xi)) \le h^{\nu}$  
then for every  $\alpha \in \partial \Omega$
the boundary cone $C(\alpha) =  C(\alpha,\Gamma (2m)^2h^{\nu}, \theta, -\vec{n}_{\alpha})$,  
has the norming set property:
$$
\forall p\in \Pi_{2m}, 
\forall \alpha \in \partial \Omega, \ 
\quad 
    \|p\|_{L_{\infty}(C(\alpha))}
    \le 
    2 \left\|p\left|_{\Xi\cap C(\alpha)}\right.\right\|_{\ell_{\infty}(\Xi\cap C(\alpha))}.
$$
As in Section \ref{SS:Error_kernels}, we have that
$|\Lambda_jp(\alpha) |\le C_2\|p\left|_{\Xi\cap C(\alpha)}\right.\|_{\ell_{\infty}}$.
This is sufficient to ensure that  boundary kernels $a_j:\partial \Omega\times \Xi\to \reals$ 
exist 
so that the following three properties hold.
Namely, 
\begin{enumerate}
\item 
$\sum_{\xi\in\Xi}a_j(\alpha,\xi) p (\xi) = \lambda_{j}p(\alpha)$ for all $p\in \Pi_{2m}$,
 \item 
 $|\alpha - \xi| >\Gamma (2m)^2 h^{\nu}$ implies $a_j(\alpha,\xi) = 0$ 
\item 
$\max_{\alpha\in \partial \Omega} \sum_{\xi\in\Xi} |a_j(\alpha,\xi)| \le K h^{-\nu j}$, with $K$ depending only on $m$ and $\Omega$.
\end{enumerate}
Consequently,
$$E_j(x,\alpha) = |\lambda_{j,\alpha}\phi(x-\alpha) - k_j(x,\alpha)| \le 
C h^{\nu(2m-d-j)} \left(1+\frac{\dist(x,\alpha)}{h^{\nu}}\right)^{-(d+1)}$$
and the corresponding operator has norm
$\|\mathcal{E}_j\|_{p\to p} \le Ch^{\nu(2m-j-1+1/p)}$. This ensures the following theorem
 %
 %
 \begin{theorem} \label{overcoming}
  There are positive constants $h_0$ and $C$ (depending on $\Omega$ and $m$) so that for all 
  $\Xi$ with fill distance $h\le h_0$,
  and satisfying the extra condition 
   $\max_{x\in \partial \Omega }\dist(x,(\Xi\cap\Omega_{h,\nu})) \le h^{\nu}$,
if $f\in W_p^{2m}(\Omega)$ (or $C^{2m}(\overline{\Omega})$ in case $p=\infty$)
then
 $$\|f - T_{\Xi} f\|_{L_p(\Omega)} \le C\left(h^{2m}  \|\Delta^m f\|_{L_p(\Omega)}
 +\sum_{j=0}^{m-1} h^{\nu(2m-j-1+\frac1{p})} \| f\|_{W_{p}^{2m}(\Omega)}\right).$$
 \end{theorem}
 %
 \begin{proof}
 The result follows from the argument used in Lemma \ref{approximation_by_operator}. The details are left to the reader.
 \end{proof}
 %
 This indicates how we may ``oversample'' $\Xi$.
 For $1\le p\le \infty$, let $\nu = \frac{2mp}{mp+1}$ (or $\nu =2$ when $p=\infty$). 
 This is the critical exponent that delivers $L_p$ approximation order $2m$.
 Then the lowest order term in Theorem
 \ref{overcoming} is controlled by $h^{2m}$ and
 $$\|f - T_{\Xi} f\|_{L_p(\Omega)} \le C h^{2m}  \| f\|_{W_p^{2m}(\Omega)}.$$
 
\paragraph{Selecting points in $\Omega_{h,\nu}$} 
To accomplish this practically, given a set of centers $\Xi\subset \Omega$ with fill distance $h$,
we  sample points $\Xi_{\partial,0}$ on $\partial \Omega$ with a density of 
 $\max_{x\in \partial \Omega} \dist(x,\Xi_{\partial,0}) =h^{\nu}$. 
 Extend this into $\Omega$ by choosing
 $2m$ layers of the form $\Xi_{\partial,j} = \{\xi^* = \xi + jh^{\nu} \vec{n}_{\xi}\mid \xi \in \Xi_{\partial,0}\}$. In that
 case, we have (for sufficiently small $h$) that $\bigcup_{j=0}^{2m} \Xi_{\partial, j}$ is a norming set for $\Omega_{h,\nu} = \{x\in \Omega\mid
 \dist(x,\partial \Omega)\le 2m h^{\nu}\}$.
 
 \paragraph{When is it feasible?} A set of centers $\Xi\subset \Omega$ has cardinality 
 $\#\Xi \ge C \mathrm{vol}(\Omega) h^{-d}$.  
 Since $\#\Xi_{\partial,j}\sim \#\Xi_{\partial,0}\sim Ch^{-\nu(d-1)}$, 
 the set of additional points $\bigcup_{j=0}^{2m} \Xi_{\partial, j}$ may have cardinality bounded by
 $(\#\Xi_{\partial,0})(2m+1)\le C m h^{-\nu(d-1)}.$
If we desire that the supplementary points do not exceed $C h^{-d}$ asymptotically 
(meaning that the number or extra centers required to achieve approximation order ${2m}$ is kept on par with the number of original centers), then 
for fixed $d$,  the $L_p$ approximation order $2m$  can be achieved for $1\le p\le \frac{d}{(d-2)m}$ without increasing (asymptotically) the number of centers.

\section{Beppo-Levi Extension}\label{S:Ext}
In this section we present a boundary layer representation of Duchon's  
norm minimizing extension operator \cite{D2},
which takes functions in $W_2^m(\Omega)$ to functions in the Beppo-Levi space 
$$D^{-m}L_2(\reals^d) = \{g\in W_{2,loc}^m\mid \forall |\beta|=m,\, D^{\beta}g \in L_2(\reals^d)\}.$$
This is the semi-Hilbert space of functions whose $m$th derivatives are globally $L_2$ equipped with 
 the Beppo-Levi semi-norm 
$$|g |_{D^{-m}L_2(\reals^d)} = 
\left(\sum_{|\alpha| = m} \begin{pmatrix} m\\ \alpha\end{pmatrix} \int_{\reals^d}|D^{\alpha} g(x)|^2 \dif x\right)^{1/2}.$$
Since $m>d/2$, $D^{-m}L_2(\reals^d)\subset C(\reals^d)$, by the Sobolev embedding theorem 
(this follows from the chain of embeddings $D^{-m}L_2(\reals^d)\subset W_{2,loc}^m(\reals^d) \subset C(\reals^d)$).
We consider the extension operator  $E_1 :W_2^m(\Omega)\to D^{-m}L_2(\reals^d): f\mapsto f_e$ 
that minimizes the Beppo-Levi semi-norm 
$$E_1 f =f_e := \mathrm{argmin} \{|g|_{D^{-m}L_2(\reals^d)} \, \mid \,g\left|_{\Omega} \right.= f \}$$
Such extensions can be written as 
$$f_e = \phi*\mu_f +\tp$$ with
$\tp$ a polynomial in $\Pi_{m-1}$ and
$\mu_f$ a distribution supported in $\overline{\Omega}$ that annihilates $\Pi_{m-1}$.
Unfortunately, not much more can be said about $\mu_f$ or $p$.

\subsection{Extension of functions in $W_2^m(\reals^d)$}
For $f\in W_2^m(\reals^d)$, Corollary \ref{Dirichlet_corollary} ensures that the solution to (\ref{Dirichlet}) with $h_k = \lambda_k f\in W_2^{m-k- 1/2}(\partial \Omega)$
satisfies $f_1 = \sum_{j=0}^{m-1} V_j g_j +p$, and that this functions lies in $W_2^m(\Omega)$. Because $g_j \in W_2^{j+1/2-m}(\partial\Omega)$,
we have that $\Lambda_j^{\t}(g_j \cdot \delta_j)\in W_2^{-m}(\reals^d)$, and $f_1 = \sum_{j=0}^{m-1}\Lambda_j^{\t}(g_j \cdot \delta_j)* \phi + p$.

The remainder $f_2 = f-f_1$ satisfies  $\lambda_j(f-f_1) = 0$ for $j=0\dots m-1$. In other words, its Dirichlet data vanishes, 
and the zero extension of $f-f_1$, denoted
$(f-f_1)_z$, lies in $W_2^m(\reals^d)$.
Therefore, $\Delta^m (f-f_1)_z\in W_2^{-m}(\reals^d)$, and it has support in $\overline{\Omega}$.
 
 We define $\nu_f := \Delta^m (f-f_1)_z + \sum_{j=0}^{m-1} \Lambda_j^{\t}(g_j \cdot \delta_j)$, and note that $f\mapsto \nu_f$ is bounded from
 $W_2^m(\Omega)$ to $W_2^{-m}(\reals^d)$. Consequently, $\nu_f*\phi \in W_{2,loc}^m(\reals^d)$.

 To guarantee that  $\nu_f*\phi$ resides in $D^{-m} L_2$, we need to demonstrate a polynomial annihilation property of $\nu_f$. 
 This is done
 below in Lemma \ref{annihilation}.
The result then follows from the
 fact that 
  for $|\alpha|=m$, $D^{\alpha}\nu_f$ is a compactly supported distribution that annihilates polynomials of degree $2m-1$,
  and therefore 
 $D^{\alpha} \nu_f *\phi =(D^{\alpha} \nu_f)*\phi$. 
From Lemma \ref{conv_with_FS} 
we have that $|D^{\alpha} \nu_f *\phi (x)|\le C(1+ |x|)^{-d}$ which shows that $D^{\alpha} \nu_f *\phi (x)\in L_2(\reals^d)$ globally.

\begin{lemma} \label{annihilation}
For $q\in \Pi_{m-1}$, $\langle \nu_f , q\rangle = 0$.
\end{lemma}
\begin{proof}
We have that
$
\langle \nu_f,q\rangle
= 
\langle (f-f_1)_z, \Delta^m q\rangle
+ 
 \sum_{j=0}^{m-1} \langle g_j, \lambda_j q\rangle
$.
 Because $q\in \Pi_{m-1}$,
$\Delta^m q = 0$, and
employing the side conditions $P^{\t} \bfg = 0$ shows that the final sum vanishes.
\end{proof}

We are now ready to prove the  main theorem for this section. Its proof is similar to that of Lemma  \ref{oneone}.
\begin{theorem}
For $f\in W_2^{m}(\overline{\Omega})$,
$f_e = \nu_f *\phi+p$.
\end{theorem}
\begin{proof}
We write $f_e = \mu_f + \tp$, and
let  $F= (\nu_f-\mu_f)*\phi+p-\tp.$ Note that $F \in D^{-m}L_2(\reals^d)$. 
For $\mu_f* \phi+\tp$ this is clear (it is the least extension in $D^{-m} L_2(\reals^d)$), while for $\nu_f*\phi+p$ it has been shown above.  

Observe that  $\Delta^m F(x) = 0 $ for $x\in \reals^d\setminus \partial \Omega$. 
Indeed, $F =0$ inside $\Omega$, because this is where both extension operators equal $f$.

  We focus on $\reals^d\setminus\overline{ \Omega}$, where $F$ is smooth, thanks to the fact  that  $\nu_f$ and $\mu_f$ are both supported in $\overline{\Omega}$.
Here $F\in W_{2,loc}^m(\reals^d)$  satisfies the $m$-fold Laplace equation $\Delta^m F(x) = 0$  with homogeneous Dirichlet conditions $\lambda_jF =0,$ for $ j=0\dots m-1$.
The polynomial annihilation property $(\nu_f -\mu_f)\perp \Pi_{m-1}$ 
in conjunction with Lemma \ref{conv_with_FS}
implies that 
$D^{\beta} F(x) \le C (1+|x|)^{m-1 -|\beta|}$, which means that $F = 0$ in $\reals^d\setminus \Omega$. 
Since $F\in C(\reals^d)$  this implies hat $F = 0$ throughout $\reals^d$. 

Finally, this implies that $(\nu_f-\mu_f)*\phi\in \Pi_{m-1}$. 
Since $\nu_f-\mu_f$ is supported in $\overline{\Omega}$, 
$\widehat{\nu_f-\mu_f}$ is entire, and simultaneously supported at $\{0\}$. Thus,
$\nu_f=\mu_f$ and $p=\tp$.
\end{proof}
%
%
%
\subsection{Extension of functions in $W_2^{2m}(\Omega)$}
If $f:\Omega \to \reals$ has greater smoothness, we may be able to say more about the distribution $\nu_f$. 
In that case, we use the extended representation (\ref{Main_Rep}) given by Theorem \ref{Main_Rep_Theorem}.
Namely, we have that 
\begin{equation*}
\nu_f *\phi
=
\left(\bigl(\Delta^m f\bigr)_z
+\sum_{j=0}^{m-1} \bigl(\Lambda_j^{\t}(N_jf\cdot\delta_{\partial\Omega}   )\bigr) \right)*\phi
\end{equation*}
Note that this does not hold in general for $f\in W_2^m(\reals^d)$. 
In particular,  the operators $N_j$ (which have a component the higher order boundary operator $\lambda_{2m-j-1}$),
are only defined on $B_{2,1}^{2m -j-1/2}(\Omega)$.

Note that for $f\in W_2^{2m}(\Omega)$, 
$ \Delta^m f$ is in $L_2(\Omega)$, and its zero extension $(\Delta^m f)_z$ is in $L_2(\reals^d)$. 
Similarly, $N_j f\in L_2(\partial \Omega)$, 
and $N_jf\cdot \delta_{\partial \Omega}\in W_2^{-1/2-\epsilon}(\reals)$ for any $\epsilon>0$.
From this, we obtain $\Lambda_j^{\t} (N_jf\cdot \delta_{\partial \Omega})\in W_2^{-1/2 - j - \epsilon}(\reals^d)\subset W_2^{-m}(\reals^d)$, provided
$\epsilon\le 1/2$ (because $j\le m-1$).

\bibliographystyle{siam}
\bibliography{LayerPots}

\end{document}